\documentclass[hidelinks,onefignum,onetabnum]{siamart251216}

\usepackage{lipsum}
\usepackage{amsfonts}
\usepackage{graphicx}
\usepackage{epstopdf}

\usepackage{algorithmic}
\ifpdf
  \DeclareGraphicsExtensions{.eps,.pdf,.png,.jpg}
\else
  \DeclareGraphicsExtensions{.eps}
\fi
\usepackage{empheq}
\usepackage{amsmath}

\newsiamremark{remark}{Remark}
\newsiamremark{hypothesis}{Hypothesis}
\crefname{hypothesis}{Hypothesis}{Hypotheses}
\newsiamthm{claim}{Claim}
\newsiamremark{fact}{Fact}
\crefname{fact}{Fact}{Facts}

\headers{Arbitrarily High-Order SP-PFEM for geometric PDEs}{Weizhu Bao, Yifei Li, and Dongmin Wang }

\title{Arbitrarily High-Order Structure-Preserving Parametric Finite Element Methods for Geometric PDEs via Local Pullback Flow Maps
\thanks{Submitted to the editors Date.
}}

\author{
Weizhu Bao\thanks{Department of Mathematics, National University of Singapore,
Singapore 119076 (\email{matbaowz@nus.edu.sg}, \email{e1583440@u.nus.edu}).}
\and
Yifei Li\thanks{Mathematisches Institut, Universit{\"a}t T{\"u}bingen, Auf der Morgenstelle 10, 72076 T{\"u}bingen, Germany (\email{yifei.li@mnf.uni-tuebingen.de})}
\and
Dongmin Wang\footnotemark[2]
}

\usepackage{amsopn}

\usepackage{amssymb}

\ifpdf
\hypersetup{
  pdftitle={Structure-Preserving Parametric Finite Element Methods via Local Pullback Flow Maps},
  pdfauthor={}
}
\fi

\begin{document}

\maketitle

\begin{abstract}
We develop arbitrarily high-order structure-preserving parametric finite element methods for geometric PDEs including surface diffusion and volume-preserving mean curvature flow based on local pullback flow maps. The central idea is to reformulate the geometric PDEs on an arbitrary intermediate hypersurface and represent the subsequent evolution through a local pullback flow map. Using admissible local pullbacks, we derive two formulations from the corresponding pullback identities for the Laplace–Beltrami operator, referred to as the conformal and equidistribution formulations. We discretize both formulations using arbitrary-degree isoparametric finite elements and Runge--Kutta collocation methods with positive weights. For the conformal formulation, weighted averages of the normal--Jacobian vector ensure exact preservation of the enclosed area or volume, while algebraic stability of the Runge--Kutta method additionally guarantees dissipation of the perimeter or surface area. For the equidistribution formulation, a path-averaged normal--Jacobian vector and a geometric discrete gradient guarantee both exact preservation of the enclosed area or volume and dissipation of the perimeter or surface area, without requiring algebraic stability. Numerical experiments confirm the expected high-order accuracy and structure-preserving properties.
\end{abstract}



\section{Introduction}

Geometric partial differential equations (PDEs) or flows arise in a wide range of applications in materials science, physics, multiphase flows, and related fields.
Typical examples include surface diffusion (SD) and volume-preserving mean curvature flow (VP-MCF) 
as well as many other geometric flows. In fact, surface diffusion describes mass transport along material surfaces driven by curvature variations and plays an important role in thermal grooving \cite{Mullins1957}, epitaxial growth \cite{KesslerLevineSander1992}, and solid-state dewetting of thin films \cite{Thompson2012,WangJiangBaoSrolovitz2015,BaoJiangWangZhao2017,ZhaoJiangBao2020}. 
Volume-preserving mean curvature flow and related constrained curvature flows arise in attachment-limited kinetics and coarsening \cite{MugnaiSeis2013,MugnaiSeisSpadaro2016}, phase separation \cite{RubinsteinSternberg1992,BronsardStoth1997,ChenHilhorstLogak2010}, shape recovery and image processing \cite{CapuzzoDolcettaFinziVitaMarch2002}, and constrained shape control \cite{LaurainWalker2021}.

Let $\Gamma(t)\subset\mathbb R^d$ with $d\in\{2,3\}$ be a smooth closed hypersurface with outward unit normal $\boldsymbol n$, mean curvature $\mathcal H$, and normal velocity $V$. The two geometric flows are governed by
\begin{equation}
V= \begin{cases} \Delta_{\Gamma(t)}\mathcal H, & \text{for SD},\\ 
-\mathcal H+\overline{\mathcal H}, & \text{for VP-MCF},
\end{cases}
\qquad
\overline{\mathcal H}:=\frac{1}{|\Gamma(t)|}\int_{\Gamma(t)}\mathcal H\,\mathrm d\Gamma(t),
\label{eq:intro-geometric-flows}
\end{equation}
where $\Delta_{\Gamma(t)} = \nabla_{\Gamma(t)} \cdot \nabla_{\Gamma(t)}$ denotes the Laplace--Beltrami operator on $\Gamma(t)$. It is well known that both flows possess two fundamental geometric properties: (i) preservation of the area/volume enclosed by the closed curve/surface, and (ii) dissipation of the total perimeter/surface area \cite{BarrettGarckeNuernberg2007,BarrettGarckeNurnberg2007b}. These geometric structures provide important criteria for the design of stable and accurate numerical methods.

Parametric finite element methods (PFEMs) have been widely used for geometric evolution equations including solid-state dewetting \cite{ZhaoJiangBao2021,BaoZhao2023}. An early PFEM for mean curvature flow was proposed by Dziuk \cite{Dziuk1991}. B\"ansch et al.\ later developed a parametric finite element method for surface diffusion based on a mixed variational formulation \cite{BanschMorinNochetto2005}. Barrett, Garcke, and N\"urnberg (BGN) established a systematic PFEM framework for a wide class of geometric flows \cite{BarrettGarckeNuernberg2007,BarrettGarckeNurnberg2007b,BarrettGarckeNuernberg2008}. With a parametrization $\boldsymbol X(\cdot,t)$ of $\Gamma(t)$, the BGN formulation couples the normal-velocity equation with the curvature identity,
\begin{equation}
\boldsymbol n\cdot\partial_t\boldsymbol X=V,
\qquad
\mathcal H\boldsymbol n=-\Delta_{\Gamma(t)}\boldsymbol X,
\label{eq:intro-bgn}
\end{equation}
and its discretization combines energy stability with favorable tangential mesh redistribution. More recently, Bao and Zhao \cite{BaoZhao2021} proposed a first-order structure-preserving PFEM for surface diffusion by introducing a time-weighted normal between two consecutive hypersurfaces, which yields exact preservation of the enclosed area or volume together with energy dissipation. This structure-preserving technique has since been applied to a broader range of geometric PDEs \cite{BaoLi2024Anisotropic, BaoLi2025Unified, BaoGarckeNuernbergZhao2023}.

Motivated by the limited temporal accuracy of such first-order schemes, higher-order approximations of geometric flows have received considerable attention in recent years. In space, Garcke et al.\ \cite{GarckeNuernbergPraetoriusZhang2025} developed arbitrary-degree isoparametric finite element methods for mean curvature flow and surface diffusion. In time, Jiang et al.\ proposed several higher-order BGN-based schemes, including second-order Crank--Nicolson--leapfrog, higher-order BDF, and predictor--corrector methods \cite{JiangSuZhang2024SecondOrder,JiangSuZhang2024BDF,JiangSuZhangZhang2025}. On the analytical side, high-order and optimal-order error estimates are available for several geometric flows and related problems \cite{KovacsLiLubich2019,KovacsLiLubich2021,HuLi2022,Li2026convergence}.

These developments improve the accuracy and convergence theory, but higher-order accuracy does not by itself ensure preservation of the underlying geometric structures. Recent studies have therefore increasingly focused on high-order structure-preserving discretizations. Garcke et al.\ \cite{GarckeJiangSuZhang2025} introduced Lagrange-multiplier formulations to enforce volume preservation and energy dissipation within higher-order time discretizations. Pan et al.\ \cite{PanJiaZhang2026} obtained second-order structure-preserving PFEMs by exploiting exact geometric identities along a quadratic temporal path, while Zhang et al.\ \cite{ZhangAndrewsFarrell2026} constructed arbitrary-order schemes using a dual-MDR formulation and continuous Petrov--Galerkin time discretization. Related developments can also be found in \cite{LiGuoBi2025,GuoLi2025,GaoLi2025,DuanYang2026}. Building on these advances, the construction of arbitrarily high-order structure-preserving and energy-stable parametric finite element methods remains an important problem.

In this paper, we develop arbitrarily high-order structure-preserving PFEMs for SD and VP-MCF based on local pullback flow maps. The main obstacle lies in the temporal discretization. High-order time discretizations, such as Runge--Kutta methods, generally require the equations to be posed on a fixed domain, whereas for geometric flows the domain is the evolving hypersurface itself. Moreover, the curvature identity
$\mathcal H\boldsymbol n=-\Delta_{\Gamma(t)}\boldsymbol X$ in
\eqref{eq:intro-bgn} involves the Laplace--Beltrami operator on the current
hypersurface $\Gamma(t)$. To pose \eqref{eq:intro-bgn} on a common fixed domain and develop high-order structure-preserving methods, the Laplace--Beltrami operator $\Delta_{\Gamma(t)}$ has to be pulled back.

For curves, a natural way to do so is the DeTurck trick, closely related to
the equidistribution condition \cite{ElliottFritz2017}: for
$\boldsymbol X(\cdot,t):[0,1]\to\mathbb R^2$ satisfying
$\partial_\rho|\partial_\rho\boldsymbol X|=0$, one has
$\Delta_{\Gamma(t)}\boldsymbol X
=|\Gamma(t)|^{-2}\partial_{\rho\rho}\boldsymbol X$.
We express this structure locally through the pullback flow map
$\boldsymbol X_{s,t}:\Gamma(s)\to\Gamma(t)$. Equidistribution is then
equivalent to $\partial_\sigma|\partial_\sigma\boldsymbol X_{s,t}|=0$, so
the local Jacobian $J:=|\partial_\sigma\boldsymbol X_{s,t}|$ is spatially
constant. Consequently, the BGN system can be written on the fixed reference
curve $\Gamma(s)$ as
\begin{equation*}
\boldsymbol n\cdot\partial_t\boldsymbol X_{s,t}=V,
\qquad
J^2\mathcal H\boldsymbol n
=
-\Delta_{\Gamma(s)}\boldsymbol X_{s,t},
\qquad
\text{on }\Gamma(s)\times[s,T],
\end{equation*}
with $V$, $\mathcal H$, and $\boldsymbol n$ understood as their pullbacks to
$\Gamma(s)$; see \eqref{eq:sd-conformal-local}. This fixed-domain
formulation makes high-order time discretization directly applicable.

The simple pullback structure of curve equidistribution motivates us to
introduce a class of \emph{admissible local pullbacks}, namely local flow
maps equipped with parametrization conditions that yield similarly
tractable pullback identities in higher dimensions. We classify these
admissible local pullbacks in
\cref{thm:parametrization-classification}. For curves, equidistribution is
the only admissible condition, while for general surface evolutions the
relevant admissible mechanisms are conformality and equidistribution.
Equidistribution for curves and conformality for surfaces yield a common
scalar pullback identity for the Laplace--Beltrami operator, leading to the
unified \emph{conformal formulation} for $d\in\{2,3\}$. Equidistribution
in both dimensions instead yields a common divergence-form pullback
identity, leading to the unified \emph{equidistribution formulation}.

We then combine both local pullback formulations with arbitrary-degree
isoparametric finite elements
\cite{GarckeNuernbergPraetoriusZhang2025} and Runge--Kutta collocation
methods to obtain arbitrarily  high-order structure-preserving PFEMs. For the conformal
formulation, suitable averages of the normal--Jacobian vector yield exact
preservation of the enclosed area or volume, while algebraic stability
additionally guarantees energy dissipation; see
\cref{thm:direct-rk-structure}. For the
equidistribution formulation, a path-averaged normal--Jacobian vector
together with a geometric discrete gradient yields both properties without
requiring algebraic stability; see \cref{thm:equi-rk-structure}. The schemes are also extended to VP-MCF, while the underlying framework can potentially be adapted to a broader class of geometric PDEs; see, e.g., \cite{BarrettGarckeNuernberg2007,BarrettGarckeNuernberg2020,
BaoLi2024Anisotropic,BaoLiWang2025Willmore}. 

The rest of the paper is organized as follows. \Cref{sec:parametrization-framework} develops the admissible local pullback framework and its classification. \Cref{sec:surface-diffusion} derives the conformal and equidistribution local pullback formulations for SD. \Cref{sec:direct-rk} constructs their structure-preserving Runge--Kutta collocation discretizations. \Cref{sec:vpmcf-extension} extends the methods to VP-MCF. Numerical experiments are presented in \cref{sec:numerical-tests}, followed by conclusions in \cref{sec:conclusions}.
\section{Local pullback geometry}
\label{sec:parametrization-framework}

This section introduces the local pullback framework under admissible parametrization conditions.

\subsection{Geometric setting}
\label{sec:basic-notation}

Let $\mathcal S\subset\mathbb R^d$ be a smooth, connected, closed hypersurface with outward unit normal $\boldsymbol n_{\mathcal S}$. We use $\nabla_{\mathcal S}$ and $\Delta_{\mathcal S}$ for the surface gradient and Laplace--Beltrami operator, respectively; both act componentwise on vector fields, see \cite{BarrettGarckeNuernberg2020} for details. For $d=2$, with arc length $\sigma$ and unit tangent $\boldsymbol\tau_{\mathcal S}$, we have $\nabla_{\mathcal S}u=(\partial_\sigma u)\boldsymbol\tau_{\mathcal S}$ and $\Delta_{\mathcal S}u=\partial_{\sigma\sigma}u$.

We use the standard Sobolev spaces $L^2(\mathcal S)$ and $H^1(\mathcal S)$, together with the inner products
\begin{equation}\label{eq:reference-l2-product} 
(u,v)_{\mathcal S}:=\int_{\mathcal S}uv\,\mathrm d\mathcal S, 
\qquad 
\langle\boldsymbol A,\boldsymbol B\rangle_{\mathcal S}:=\int_{\mathcal S}\boldsymbol A:\boldsymbol B\,\mathrm d\mathcal S. 
\end{equation} 
For vector-valued functions, $uv$ is replaced by $\boldsymbol u\cdot\boldsymbol v$, while for matrices $\boldsymbol A:\boldsymbol B=\operatorname{tr}(\boldsymbol A^T\boldsymbol B)$. The corresponding $L^2$ norms for scalar, vector, and matrix-valued functions are all denoted by $\|\cdot\|_{L^2(\mathcal S)}$. 

For $d\in \{2,3\}$, let $\{\Gamma(t)\}_{t\in[0,T]}$ be a smooth family of connected, closed hypersurfaces in $\mathbb R^d$, with outward unit normal $\boldsymbol n_{\Gamma(t)}$. We parameterize $\Gamma(t)$ by smooth, orientation-preserving diffeomorphisms
\begin{equation}\label{eq:global-parametrization}
\boldsymbol X_t:=\boldsymbol X(\cdot,t):\Gamma(0)\to\Gamma(t),
\qquad
\boldsymbol X_0=\operatorname{id}_{\Gamma(0)},
\end{equation}
and denote the mean curvature by $\mathcal H_{\Gamma(t)}$. When no ambiguity arises, we write $\boldsymbol n:=\boldsymbol n_{\Gamma(t)}$ and $\mathcal H:=\mathcal H_{\Gamma(t)}$.\subsection{Local pullback}
\label{sec:abstract-evolution-family}

For $0\leq s\leq t\leq T$, define the local pullback by the flow map
\begin{equation}
\boldsymbol X_{s,t}:=\boldsymbol X_t\circ\boldsymbol X_s^{-1}:
\Gamma(s)\to\Gamma(t),
\qquad
\boldsymbol X_{s,s}=\operatorname{id}_{\Gamma(s)}.
\label{eq:surface-transition-map}
\end{equation}
Let $\boldsymbol P_{\Gamma(s)}:=\boldsymbol I-\boldsymbol n_{\Gamma(s)}\otimes\boldsymbol n_{\Gamma(s)}$ denote the tangential projection. We define the local pullback metric by
\begin{equation}
\boldsymbol G_{s,t}:=(\nabla_{\Gamma(s)}\boldsymbol X_{s,t})^T\nabla_{\Gamma(s)}\boldsymbol X_{s,t}+\boldsymbol n_{\Gamma(s)}\otimes\boldsymbol n_{\Gamma(s)}.
\label{eq:relative-metric}
\end{equation}
On $T_p\Gamma(s)$, $\boldsymbol G_{s,t}$ is the tangential right Cauchy--Green deformation tensor associated with the local pullback $\boldsymbol X_{s,t}$ \cite{hughes1983mathematical}, while $\boldsymbol G_{s,t}\boldsymbol n_{\Gamma(s)}=\boldsymbol n_{\Gamma(s)}$. Hence $\boldsymbol G_{s,t}$ is symmetric positive definite and maps tangent vectors of $\Gamma(s)$ to tangent vectors. The associated local surface Jacobian, i.e., the local ratio of surface measures under $\boldsymbol X_{s,t}$, is
\begin{equation}
J_{s,t}
=\sqrt{\det\boldsymbol G_{s,t}}.
\label{eq:metric-determinant-jacobian}
\end{equation}
We decompose the pullback metric into size and shape,
\begin{equation}
\begin{aligned}
&\boldsymbol G_{s,t}\boldsymbol P_{\Gamma(s)}
=J_{s,t}^{2/(d-1)}\widehat{\boldsymbol G}_{s,t}\boldsymbol P_{\Gamma(s)}, \\
&\widehat{\boldsymbol G}_{s,t}:=J_{s,t}^{-2/(d-1)}(\nabla_{\Gamma(s)}\boldsymbol X_{s,t})^T\nabla_{\Gamma(s)}\boldsymbol X_{s,t}+\boldsymbol n_{\Gamma(s)}\otimes\boldsymbol n_{\Gamma(s)},
\end{aligned}
\label{eq:relative-metric-scale-shape-decomposition}
\end{equation}
so that $\det\widehat{\boldsymbol G}_{s,t}=1$. Here $J_{s,t}^{2/(d-1)}$ describes the local change of measure and $\widehat{\boldsymbol G}_{s,t}$ the anisotropic shape distortion.
For $d=2$, the tangent space is one-dimensional, and $\widehat{\boldsymbol G}_{s,t}=\boldsymbol I$.

\begin{proposition}[Weighted pullback of the Laplace--Beltrami operator]
\label{prop:weighted-laplace-pullback}
For every $u\in C^2(\Gamma(t))$ and $\widehat u:=u\circ\boldsymbol X_{s,t}$,
\begin{equation}
J_{s,t}(\Delta_{\Gamma(t)}u)\circ\boldsymbol X_{s,t}
=\nabla_{\Gamma(s)}\cdot\bigl(J_{s,t}\boldsymbol G_{s,t}^{-1}\nabla_{\Gamma(s)}\widehat u\bigr).
\label{eq:weighted-laplace-pullback}
\end{equation}
\end{proposition}

This is the Piola identity \cite{hughes1983mathematical} applied to the vector field $\nabla_{\Gamma(t)}u$, since the chain rule gives $\nabla_{\Gamma(s)}\widehat u=(\nabla_{\Gamma(s)}\boldsymbol X_{s,t})^T(\nabla_{\Gamma(t)}u)\circ\boldsymbol X_{s,t}$, and hence $(\nabla_{\Gamma(s)}\boldsymbol X_{s,t})\boldsymbol G_{s,t}^{-1}\nabla_{\Gamma(s)}\widehat u=(\nabla_{\Gamma(t)}u)\circ\boldsymbol X_{s,t}$.

Thus the pullback is governed by the coefficient $J_{s,t}\boldsymbol G_{s,t}^{-1}$ on tangent vectors. If $J_{s,t}\boldsymbol G_{s,t}^{-1}\boldsymbol P_{\Gamma(s)}=c\,\boldsymbol P_{\Gamma(s)}$ with a constant $c>0$, then \eqref{eq:weighted-laplace-pullback} and $\Delta_{\Gamma(s)}=\nabla_{\Gamma(s)}\cdot\nabla_{\Gamma(s)}$ give
\begin{equation}
J_{s,t}(\Delta_{\Gamma(t)}u)\circ\boldsymbol X_{s,t}=c\,\Delta_{\Gamma(s)}(u\circ\boldsymbol X_{s,t}).
\label{eq:weighted-laplace-constant}
\end{equation}
Conversely, if \eqref{eq:weighted-laplace-constant} holds for all $u\in C^2(\Gamma(t))$, comparing the second-order terms gives $J_{s,t}\boldsymbol G_{s,t}^{-1}\boldsymbol P_{\Gamma(s)}=c\,\boldsymbol P_{\Gamma(s)}$, since $\boldsymbol G_{s,t}^{-1}$ is symmetric. By \eqref{eq:relative-metric-scale-shape-decomposition}, we have
\begin{equation}
J_{s,t}\boldsymbol G_{s,t}^{-1}\boldsymbol P_{\Gamma(s)}=J_{s,t}^{(d-3)/(d-1)}\,\widehat{\boldsymbol G}_{s,t}^{-1}\boldsymbol P_{\Gamma(s)}.
\label{eq:pullback-coefficient}
\end{equation}
Hence a simple pullback requires control of the size $J_{s,t}$, of the shape $\widehat{\boldsymbol G}_{s,t}$, or of both. This is achieved by the admissible conditions below.

\subsection{Admissible conditions}
\label{sec:admissible-conditions}

We use the tangential freedom to impose an admissible condition on the flow maps, generalizing equidistribution for curves.

For $p\in\Gamma(s)$, let $[\nabla_{\Gamma(s)}\boldsymbol X_{s,t}(p)]\in\mathrm{GL}^+(d-1,\mathbb R)$ denote the matrix of the restriction of $\nabla_{\Gamma(s)}\boldsymbol X_{s,t}(p)$ to $T_p\Gamma(s)\to T_{\boldsymbol X_{s,t}(p)}\Gamma(t)$ with respect to positively oriented orthonormal bases of $T_p\Gamma(s)$ and $T_{\boldsymbol X_{s,t}(p)}\Gamma(t)$. Then $[\nabla_{\Gamma(s)}\boldsymbol X_{s,t}(p)]^T[\nabla_{\Gamma(s)}\boldsymbol X_{s,t}(p)]$ is the matrix of the restriction of $\boldsymbol G_{s,t}(p)$ to $T_p\Gamma(s)$ in the chosen orthonormal basis. In particular, $\det[\nabla_{\Gamma(s)}\boldsymbol X_{s,t}(p)]=J_{s,t}(p)$.

For $d = 2$, the tangent spaces are spanned by the unit tangents, and $[\nabla_{\Gamma(s)}\boldsymbol X_{s,t}(p)]$ is the scalar $|\partial_\sigma\boldsymbol X_{s,t}(p)|=J_{s,t}(p)$. The equidistribution condition $\partial_\sigma|\partial_\sigma\boldsymbol X_{s,t}|=0$ therefore states that
\begin{equation}
[\nabla_{\Gamma(s)}\boldsymbol X_{s,t}(p)]=\lambda_{s,t}\qquad\text{for all }p\in\Gamma(s),
\label{eq:curve-equi-tangent}
\end{equation}
with a constant $\lambda_{s,t}>0$ independent of $p$, which equals the length ratio $|\Gamma(t)|/|\Gamma(s)|$. This motivates the following definition.

\begin{definition}[Admissible condition]
\label{def:first-order-condition}
Let $\mathcal G$ be a closed connected subgroup of $\mathrm{GL}^+(d-1,\mathbb R)$ with
\begin{equation}
\mathrm{SO}(d-1)\subset\mathcal G\subsetneq\mathrm{GL}^+(d-1,\mathbb R).
\label{eq:group-condition}
\end{equation}
The flow maps are called $\mathcal G$-admissible if, for every $0\leq s\leq t\leq T$, there is a constant $\lambda_{s,t}>0$ independent of $p$ such that $\lambda_{r,s}\lambda_{s,t}=\lambda_{r,t}$ for $0\leq r\leq s\leq t\leq T$, and
\begin{equation}
[\nabla_{\Gamma(s)}\boldsymbol X_{s,t}(p)]\in\lambda_{s,t}\mathcal G\qquad\text{for all }p\in\Gamma(s).
\label{eq:first-order-condition}
\end{equation}
\end{definition}

Equidistribution for curves is the case $\mathcal G=\{1\}=\mathrm{SO}(1)$. Since a change of orthonormal bases multiplies $[\nabla_{\Gamma(s)}\boldsymbol X_{s,t}(p)]$ by rotations from both sides, $\mathrm{SO}(d-1)\subset\mathcal G$ makes the condition independent of coordinates. Finally, $\mathcal G\neq\mathrm{GL}^+(d-1,\mathbb R)$ excludes the trivial case.

The admissible condition is classified by the following theorem.

\begin{theorem}[Classification of $\mathcal G$]
\label{thm:parametrization-classification}
Let $\mathcal G$ satisfy the assumptions of \cref{def:first-order-condition}.

(a) If $d=2$, then $\mathcal G=\{1\}$, and the admissible condition is equidistribution,
\begin{equation}
\nabla_{\Gamma(s)}J_{s,t}=0.
\label{eq:equidistribution}
\end{equation}

(b) If $d=3$, then $\mathcal G$ is one of $\mathrm{SO}(2)$, $\mathbb R^+\times\mathrm{SO}(2):=\{\mu\boldsymbol R:\mu>0,\ \boldsymbol R\in\mathrm{SO}(2)\}$, and $\mathrm{SL}(2,\mathbb R)$. The corresponding admissible conditions are (i) similarity, i.e. $\boldsymbol G_{s,t}=\lambda_{s,t}^2\boldsymbol P_{\Gamma(s)}+\boldsymbol n_{\Gamma(s)}\otimes\boldsymbol n_{\Gamma(s)}$; (ii) conformality, i.e.
\begin{equation}
\boldsymbol G_{s,t}=J_{s,t}\boldsymbol P_{\Gamma(s)}+\boldsymbol n_{\Gamma(s)}\otimes\boldsymbol n_{\Gamma(s)}
\qquad \Longleftrightarrow \qquad
\widehat{\boldsymbol G}_{s,t}=\boldsymbol I;
\label{eq:conformal-coordinate-evolution}
\end{equation}
and (iii) equidistribution \eqref{eq:equidistribution}, respectively.
\end{theorem}

\begin{proof}
\emph{(a)} For $d=2$, $\mathrm{GL}^+(1,\mathbb R)=\mathbb R^+$, whose connected subgroups are $\{1\}$ and $\mathbb R^+$. By \eqref{eq:group-condition}, $\mathcal G=\{1\}$. The matrix $[\nabla_{\Gamma(s)}\boldsymbol X_{s,t}(p)]$ is the scalar $|\partial_\sigma\boldsymbol X_{s,t}|=J_{s,t}$, so \eqref{eq:first-order-condition} reads $J_{s,t}=\lambda_{s,t}$, i.e., \eqref{eq:equidistribution}, since $\lambda_{s,t}$ is independent of $p$.

\emph{(b)} Let $d=3$. Since a closed connected subgroup of $\mathrm{GL}^+(2,\mathbb R)$ is determined by its Lie algebra \cite{Hall2015}, it suffices to determine the Lie subalgebras of $\mathfrak{gl}(2,\mathbb R)$ containing $\mathfrak{so}(2)$. These are $\mathfrak{so}(2)$, $\mathfrak{so}(2)\oplus\mathbb R\boldsymbol I_2$, $\mathfrak{sl}(2,\mathbb R)$, and $\mathfrak{gl}(2,\mathbb R)$ \cite{Hall2015}, which are the Lie algebras of $\mathrm{SO}(2)$, $\mathbb R^+\times\mathrm{SO}(2)$, $\mathrm{SL}(2,\mathbb R)$, and $\mathrm{GL}^+(2,\mathbb R)$. The last one is excluded by \eqref{eq:group-condition}.

It remains to identify the conditions. Write $\boldsymbol A:=[\nabla_{\Gamma(s)}\boldsymbol X_{s,t}(p)]$ for short. We know $\det\boldsymbol A=J_{s,t}(p)$, and $\boldsymbol A^T\boldsymbol A$ is the restriction of the pullback metric tensor $\boldsymbol G_{s,t}(p)$ to $T_p\Gamma(s)$ in the chosen orthonormal basis. If $\mathcal G=\mathrm{SO}(2)$, then $\boldsymbol A^T\boldsymbol A=\lambda_{s,t}^2\boldsymbol I_2$ with a constant $\lambda_{s,t}$, i.e., $\boldsymbol G_{s,t}=\lambda_{s,t}^2\boldsymbol P_{\Gamma(s)}+\boldsymbol n_{\Gamma(s)}\otimes\boldsymbol n_{\Gamma(s)}$. If $\mathcal G=\mathbb R^+\times\mathrm{SO}(2)$, then $\boldsymbol A^T\boldsymbol A=\det(\boldsymbol A)\boldsymbol I_2=J_{s,t}\boldsymbol I_2$, which is \eqref{eq:conformal-coordinate-evolution}. If $\mathcal G=\mathrm{SL}(2,\mathbb R)$, then $\det\boldsymbol A=\lambda_{s,t}^2$, i.e., $J_{s,t}$ is spatially constant, which is \eqref{eq:equidistribution}.
\end{proof}

Similarity forces the Gaussian curvatures $K$ to satisfy
$K_{\Gamma(t)}\circ\boldsymbol X_{s,t}=\lambda_{s,t}^{-2}K_{\Gamma(s)}$,
which fails for general evolutions. Thus the attainable nontrivial
conditions are equidistribution for curves and conformality or
equidistribution for surfaces.

\begin{figure}[htbp]
\centering
\includegraphics[width=0.85\linewidth]{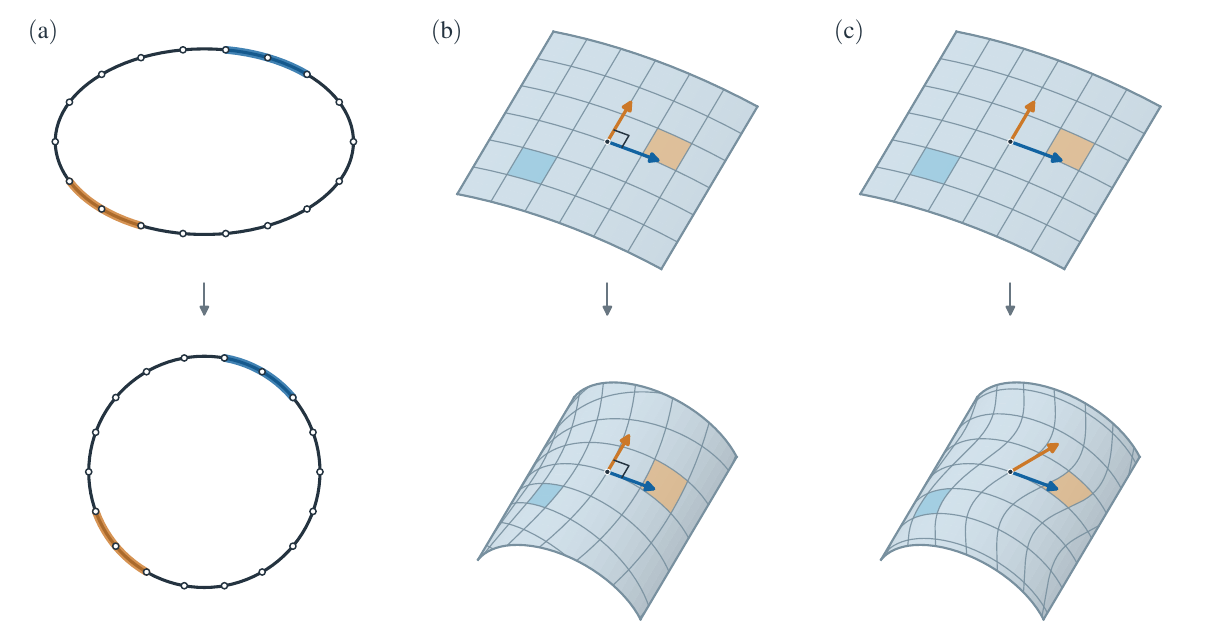}
\caption{Schematic illustration of three admissible conditions (the upper and lower rows show the initial and final configurations): (a) on an equidistributed planar curve, equal-arclength segments remain equal in length, (b) a conformal surface parametrization preserves tangent angles while allowing spatially varying area scaling, and  (c) an equidistributed surface parametrization preserves equality of initially equal-area material cells while allowing angular distortion.}
\end{figure}

\subsection{Admissible pullback of the Laplace--Beltrami operator}
\label{sec:pullback-identities}

We now evaluate \eqref{eq:weighted-laplace-pullback} under the admissible conditions. For equidistribution, we use the coefficient
\begin{equation}
\boldsymbol C_{s,t}:=J_{s,t}^2\boldsymbol G_{s,t}^{-1},
\label{eq:equi-C}
\end{equation}
which is symmetric positive definite. By \eqref{eq:pullback-coefficient}, its tangential part satisfies $\boldsymbol C_{s,t}\boldsymbol P_{\Gamma(s)}=J_{s,t}^{2(d-2)/(d-1)}\widehat{\boldsymbol G}_{s,t}^{-1}\boldsymbol P_{\Gamma(s)}$; for $d=2$, this becomes $\boldsymbol C_{s,t}\boldsymbol P_{\Gamma(s)}=\boldsymbol P_{\Gamma(s)}$.

\begin{corollary}[Equidistributed and conformal curves]
\label{cor:equi-planar-pullback}
Let $d=2$ and let $\boldsymbol X_{s,t}$ be equidistributed. Then, for every $u\in C^2(\Gamma(t))$, the conformal formulation
\begin{equation}
J_{s,t}^2(\Delta_{\Gamma(t)}u)\circ\boldsymbol X_{s,t}=\Delta_{\Gamma(s)}(u\circ\boldsymbol X_{s,t})
\label{eq:planar-equi-laplace}
\end{equation}
and the equidistribution formulation
\begin{equation}
J_{s,t}^2(\Delta_{\Gamma(t)}u)\circ\boldsymbol X_{s,t}=\nabla_{\Gamma(s)}\cdot\bigl(\boldsymbol C_{s,t}\nabla_{\Gamma(s)}\widehat u\bigr).
\label{eq:planar-equi-laplace-divergence}
\end{equation}
hold.
\end{corollary}
 
\begin{proof}
Since $J_{s,t}$ is constant by \eqref{eq:equidistribution}, multiplying \eqref{eq:weighted-laplace-pullback} by $J_{s,t}$ gives $J_{s,t}^2(\Delta_{\Gamma(t)}u)\circ\boldsymbol X_{s,t}=\nabla_{\Gamma(s)}\cdot(\boldsymbol C_{s,t}\nabla_{\Gamma(s)}\widehat u)$, and $\boldsymbol C_{s,t}\boldsymbol P_{\Gamma(s)}=\boldsymbol P_{\Gamma(s)}$ by \eqref{eq:equi-C}.
\end{proof}
 
\begin{corollary}[Conformal surfaces]
\label{cor:conformal-laplace}
Let $d=3$ and let $\boldsymbol X_{s,t}$ be conformal. Then, for every $u\in C^2(\Gamma(t))$,
\begin{equation}
J_{s,t}(\Delta_{\Gamma(t)}u)\circ\boldsymbol X_{s,t}=\Delta_{\Gamma(s)}(u\circ\boldsymbol X_{s,t}).
\label{eq:conformal-laplace}
\end{equation}
\end{corollary}
 
\begin{proof}
This is the conformal invariance of the Laplace--Beltrami operator in two dimensions \cite{Jost2006}; equivalently, \eqref{eq:pullback-coefficient} with $\widehat{\boldsymbol G}_{s,t}=\boldsymbol I$ gives \eqref{eq:weighted-laplace-constant} with $c=1$.
\end{proof}
 
\begin{corollary}[Equidistributed surfaces]
\label{cor:surface-equi-pullback}
Let $d=3$ and let $\boldsymbol X_{s,t}$ be equidistributed. Then, for every $u\in C^2(\Gamma(t))$,
\begin{equation}
J_{s,t}^2(\Delta_{\Gamma(t)}u)\circ\boldsymbol X_{s,t}
=\nabla_{\Gamma(s)}\cdot\bigl(\boldsymbol C_{s,t}\nabla_{\Gamma(s)}(u\circ\boldsymbol X_{s,t})\bigr).
\label{eq:surface-equi-laplace}
\end{equation}
\end{corollary}
 
\begin{proof}
Since $J_{s,t}$ is constant by \eqref{eq:equidistribution}, multiply \eqref{eq:weighted-laplace-pullback} by $J_{s,t}$.
\end{proof}

\begin{remark}[Genus restriction of conformal surfaces]
\label{rem:global-conformal-genus}
\Cref{cor:conformal-laplace} requires a globally conformal map $\boldsymbol X_{s,t}$. For closed genus-zero surfaces, such a map between $\Gamma(s)$ and $\Gamma(t)$ always exists by the uniformization theorem. For surfaces of positive genus, the conformal class may vary, and a global conformal map need not exist \cite{Jost2006}. We therefore restrict the conformal formulation to closed genus-zero surfaces, and use equidistribution for surfaces of positive genus.
\end{remark}

\begin{remark}[Equidistributed surfaces]
\label{rem:surface-equi-operator}
The pullback formula \eqref{eq:surface-equi-laplace} retains a divergence structure on the reference surface $\Gamma(s)$, with the coefficient $\boldsymbol G_{s,t}^{-1}\boldsymbol P_{\Gamma(s)}=J_{s,t}^{-1}\widehat{\boldsymbol G}_{s,t}^{-1}\boldsymbol P_{\Gamma(s)}$. It thus provides a possible starting point for constructing high-order discretizations under surface equidistribution. However, equidistribution alone may lead to excessive mesh stretching, since the shape distortion $\widehat{\boldsymbol G}_{s,t}$ remains uncontrolled.
\end{remark}

\section{Local pullback formulations for surface diffusion}
\label{sec:surface-diffusion}

With the above parameterization,   surface diffusion is given by
\begin{equation}\label{eq:sd}
\begin{cases}
\boldsymbol n\cdot(\partial_t\boldsymbol X_t\circ\boldsymbol X_t^{-1})
=\Delta_{\Gamma(t)}\mathcal H,\\[2mm]
\mathcal H\boldsymbol n
=-\Delta_{\Gamma(t)}\operatorname{id}_{\Gamma(t)},
\end{cases}
\qquad
\text{on }\Gamma(t),\quad t\in[0,T].
\end{equation}

We now write surface diffusion \eqref{eq:sd} on the fixed reference hypersurface $\Gamma(s)$. We abbreviate $J_{s,t}$ by $J$ and use the same symbols $\mathcal H$ and $\boldsymbol n$ for the mean curvature and unit normal on $\Gamma(t)$ and for their pullbacks $\mathcal H_{\Gamma(t)}\circ\boldsymbol X_{s,t}$ and $\boldsymbol n_{\Gamma(t)}\circ\boldsymbol X_{s,t}$ to $\Gamma(s)$; quantities on $\Gamma(s)$ carry explicit subscripts, e.g., $\boldsymbol n_{\Gamma(s)}$.

\subsection{Conformal formulation}
\label{sec:conformal-formulation}

Let $d\in\{2,3\}$, and let $\boldsymbol X_{s,t}$ be equidistributed between connected closed curves for $d=2$, or conformal between closed genus-zero surfaces for $d=3$ (cf.~Remark~\ref{rem:global-conformal-genus}). Applying \eqref{eq:planar-equi-laplace} for $d=2$ and \eqref{eq:conformal-laplace} for $d=3$, surface diffusion \eqref{eq:sd} becomes
\begin{subequations}
\label{eq:sd-conformal-local}
\begin{align}
&J^{2/(d-1)}\boldsymbol n\cdot\partial_t\boldsymbol X_{s,t}=\Delta_{\Gamma(s)}\mathcal H,\, \quad \text{on }\Gamma(s)\times[s,T], \label{eq:sd-conformal-local-normal}\\[2mm]
&J^{2/(d-1)}\mathcal H\boldsymbol n=-\Delta_{\Gamma(s)}\boldsymbol X_{s,t},\qquad \text{on }\Gamma(s)\times[s,T], \label{eq:sd-conformal-local-curvature}
\end{align}
\end{subequations}
with $\boldsymbol X_{s,s}=\operatorname{id}_{\Gamma(s)}$. Under equidistribution ($d\in\{2,3\}$), $J$ is spatially constant, and hence
\begin{equation}
J=\overline J:=\frac1{|\Gamma(s)| }\int_{\Gamma(s)}J \,\mathrm d  \Gamma(s) = \frac{|\Gamma(t)|}{|\Gamma(s)|};
\label{eq:equi-jacobian}
\end{equation}
under conformality ($d=3$), taking the trace of \eqref{eq:conformal-coordinate-evolution} gives
\begin{equation}
J=\frac12|\nabla_{\Gamma(s)}\boldsymbol X_{s,t}|^2.
\label{eq:conformal-jacobian}
\end{equation}

\begin{theorem}[Equivalence of the conformal formulation]
\label{thm:conformal-equivalence}
Let $\boldsymbol X_{s,t}:\Gamma(s)\to\Gamma(t)$ be a smooth, orientation-preserving diffeomorphism between connected closed curves ($d=2$) or closed genus-zero surfaces ($d=3$), and let $\boldsymbol X_t=\boldsymbol X_{s,t}\circ\boldsymbol X_s$. Then $(\boldsymbol X_{s,t},\mathcal H)$ satisfies \eqref{eq:sd-conformal-local} if and only if $\boldsymbol X_{s,t}$ is equidistributed for $d=2$ and  conformal for $d=3$, and $(\boldsymbol X_t,\mathcal H)$ satisfies \eqref{eq:sd}.
\end{theorem}

\begin{proof}
The reverse implication follows from the derivation above. Conversely, the tangential part of \eqref{eq:sd-conformal-local-curvature} gives $(\boldsymbol I-\boldsymbol n_{\Gamma(t)}\otimes\boldsymbol n_{\Gamma(t)})\Delta_{\Gamma(s)}\boldsymbol X_{s,t}=\boldsymbol0$, i.e., $\boldsymbol X_{s,t}$ is a harmonic map. For $d=2$, with the arc length $\sigma$ on $\Gamma(s)$, this gives $0=\partial_{\sigma\sigma}\boldsymbol X_{s,t}\cdot\partial_\sigma\boldsymbol X_{s,t}=J\partial_\sigma J$, i.e., equidistribution. For $d=3$, a harmonic diffeomorphism between closed genus-zero surfaces is conformal \cite{HeleinWood2008}. Hence \eqref{eq:planar-equi-laplace} or \eqref{eq:conformal-laplace} applies, and \eqref{eq:sd-conformal-local} is the pullback of \eqref{eq:sd}.
\end{proof}

\subsection{Unified weak conformal formulation}
\label{sec:unified-weak}

Let $d\in\{2,3\}$ and fix $s\in[0,T)$; we use the spaces and inner products of Section~\ref{sec:basic-notation} on $\Gamma(s)$. For $d=2$, let $\boldsymbol\tau$ be the counterclockwise unit tangent to $\Gamma(s)$, $(v_1,v_2)^\perp:=(-v_2,v_1)$, and $\boldsymbol n_{\Gamma(s)}=-\boldsymbol\tau^\perp$; for $d=3$, let $(\boldsymbol\tau_1,\boldsymbol\tau_2)$ be a local positively oriented orthonormal tangent frame on $\Gamma(s)$. We define the normal--Jacobian vector
\begin{equation}
\boldsymbol N(\boldsymbol X_{s,t}):=
\begin{cases}
-(\partial_{\boldsymbol\tau}\boldsymbol X_{s,t})^\perp,&d=2,\\
\partial_{\boldsymbol\tau_1}\boldsymbol X_{s,t}\times\partial_{\boldsymbol\tau_2}\boldsymbol X_{s,t},&d=3,
\end{cases}
\qquad
\boldsymbol N=J\boldsymbol n,
\label{eq:normal-jacobian-tangent-formulas}
\end{equation}
abbreviated by $\boldsymbol N$. By \eqref{eq:equi-jacobian}, $J^{2/(d-1)}\boldsymbol n=\overline J^{3-d}\boldsymbol N$. Testing \eqref{eq:sd-conformal-local} with $\phi$ and $\boldsymbol\omega$ and integrating by parts gives the unified weak formulation: find $(\boldsymbol X_{s,t},\mathcal H)\in[H^1(\Gamma(s))]^d\times H^1(\Gamma(s))$, with $\boldsymbol X_{s,s}=\operatorname{id}_{\Gamma(s)}$, such that, for $t\in(s,T]$,
\begin{subequations}
\label{eq:unified-sd-weak}
\begin{align}
&\overline J^{3-d}(\boldsymbol N\cdot\partial_t\boldsymbol X_{s,t},\phi)_{\Gamma(s)}
=-(\nabla_{\Gamma(s)}\mathcal H,\nabla_{\Gamma(s)}\phi)_{\Gamma(s)},
&&\forall\phi\in H^1(\Gamma(s)), \label{eq:unified-sd-weak-normal}\\[2mm]
&\overline J^{3-d}(\mathcal H\boldsymbol N,\boldsymbol\omega)_{\Gamma(s)}
=\langle\nabla_{\Gamma(s)}\boldsymbol X_{s,t},\nabla_{\Gamma(s)}\boldsymbol\omega\rangle_{\Gamma(s)},
&&\forall\boldsymbol\omega\in[H^1(\Gamma(s))]^d. \label{eq:unified-sd-weak-curvature}
\end{align}
\end{subequations}

Let $|\Omega(t)|$ denote the enclosed area ($d=2$) or volume ($d=3$), and $|\Gamma(t)|$ the length or area of $\Gamma(t)$.

\begin{theorem}[Preservation and dissipation]
\label{thm:unified-weak-structure}
A sufficiently smooth solution $(\boldsymbol X_{s,t},\mathcal H)$ of \eqref{eq:unified-sd-weak} satisfies
\begin{equation*}
|\Omega(t_2)|=|\Omega(t_1)|,\qquad |\Gamma(t_2)|\leq|\Gamma(t_1)|,\qquad s\leq t_1\leq t_2\leq T.
\end{equation*}
\end{theorem}

\begin{proof}
By the transport identities \cite{BarrettGarckeNuernberg2020} and a change of variables to $\Gamma(s)$, using $\boldsymbol{N} = J\boldsymbol{n} $ and $(\partial_t\boldsymbol X_t\circ\boldsymbol X_t^{-1})\circ\boldsymbol X_{s,t}=\partial_t\boldsymbol X_{s,t}$,
\begin{align}
\frac{\mathrm d}{\mathrm dt}|\Omega(t)|
&=(\boldsymbol N\cdot\partial_t\boldsymbol X_{s,t},1)_{\Gamma(s)}
=0, \label{eq:unified-volume}\\
\frac{\mathrm d}{\mathrm dt}|\Gamma(t)|
&=(\mathcal H\boldsymbol N,\partial_t\boldsymbol X_{s,t})_{\Gamma(s)}
=-\overline J^{d-3}\|\nabla_{\Gamma(s)}\mathcal H\|_{L^2(\Gamma(s))}^2\leq0, \label{eq:unified-dissipation}
\end{align}
where the last equalities follow from \eqref{eq:unified-sd-weak-normal} with $\phi=1$ and $\phi=\mathcal H$. Integrating over $[t_1,t_2]$ proves the result.
\end{proof}
\subsection{Equidistribution formulation}
\label{sec:equi-formulation}

Let $d\in\{2,3\}$, and let $\boldsymbol X_{s,t}$ be equidistributed between connected closed curves for $d=2$, or between closed surfaces of arbitrary genus for $d=3$. By \eqref{eq:planar-equi-laplace-divergence} and \eqref{eq:surface-equi-laplace}, surface diffusion \eqref{eq:sd} becomes
\begin{subequations}
\label{eq:equi-sd-local}
\begin{align}
&J^2\boldsymbol n\cdot\partial_t\boldsymbol X_{s,t}=\nabla_{\Gamma(s)}\cdot\bigl(\boldsymbol C_{s,t}\nabla_{\Gamma(s)}\mathcal H\bigr), \quad \quad \text{on }\Gamma(s)\times[s,T], \label{eq:equi-sd-local-velocity}\\[2mm]
&J^2\mathcal H\boldsymbol n=-\nabla_{\Gamma(s)}\cdot\bigl((\nabla_{\Gamma(s)}\boldsymbol X_{s,t})\boldsymbol C_{s,t}\bigr), \qquad \text{on }\Gamma(s)\times[s,T], \label{eq:equi-sd-local-curvature}
\end{align}
\end{subequations}
with $\boldsymbol X_{s,s}=\operatorname{id}_{\Gamma(s)}$. For $d=2$, \eqref{eq:equi-sd-local} coincides with \eqref{eq:sd-conformal-local}.

\begin{theorem}[Equivalence of the equidistribution formulation]
\label{thm:equi-equivalence}
Let $\boldsymbol X_{s,t}:\Gamma(s)\to\Gamma(t)$ be a smooth, orientation-preserving diffeomorphism between connected closed curves ($d=2$) or closed surfaces ($d=3$), and let $\boldsymbol X_t=\boldsymbol X_{s,t}\circ\boldsymbol X_s$. Then $(\boldsymbol X_{s,t},\mathcal H)$ satisfies \eqref{eq:equi-sd-local} if and only if $\boldsymbol X_{s,t}$ is equidistributed and $(\boldsymbol X_t,\mathcal H)$ satisfies \eqref{eq:sd}.
\end{theorem}

\begin{proof}
The reverse implication follows from the derivation above. Conversely, applying \eqref{eq:weighted-laplace-pullback} componentwise to $\operatorname{id}_{\Gamma(t)}$, with $\Delta_{\Gamma(t)}\operatorname{id}_{\Gamma(t)}=-\mathcal H_{\Gamma(t)}\boldsymbol n_{\Gamma(t)}$, and using the product rule with $\boldsymbol C_{s,t}=J[J\boldsymbol G_{s,t}^{-1}]$ gives
\begin{equation*}
\nabla_{\Gamma(s)}\cdot\bigl((\nabla_{\Gamma(s)}\boldsymbol X_{s,t})\boldsymbol C_{s,t}\bigr)
=-J^2(\mathcal H_{\Gamma(t)}\boldsymbol n_{\Gamma(t)})\circ\boldsymbol X_{s,t}+J(\nabla_{\Gamma(s)}\boldsymbol X_{s,t})\boldsymbol G_{s,t}^{-1}\nabla_{\Gamma(s)}J.
\end{equation*}
Equation~\eqref{eq:equi-sd-local-curvature} shows that the left-hand side is normal to $\Gamma(t)$. In the decomposition above, the first term on the right-hand side is normal, while the second is tangent: indeed, $\boldsymbol G_{s,t}^{-1}\nabla_{\Gamma(s)}J\in T_p\Gamma(s)$ and $\nabla_{\Gamma(s)}\boldsymbol X_{s,t}$ maps $T_p\Gamma(s)$ into $T_{\boldsymbol X_{s,t}(p)}\Gamma(t)$. Hence the tangential term must vanish,
$(\nabla_{\Gamma(s)}\boldsymbol X_{s,t})\boldsymbol G_{s,t}^{-1}\nabla_{\Gamma(s)}J=\boldsymbol0$.
Since $\boldsymbol G_{s,t}^{-1}$ is invertible on $T_p\Gamma(s)$ and $\nabla_{\Gamma(s)}\boldsymbol X_{s,t}$ is injective on this space, we obtain $\nabla_{\Gamma(s)}J=\boldsymbol0$. Thus $J$ is spatially constant on the connected hypersurface $\Gamma(s)$, so \eqref{eq:planar-equi-laplace} or \eqref{eq:surface-equi-laplace} applies, and \eqref{eq:equi-sd-local} is the pullback of \eqref{eq:sd}.
\end{proof}

\subsection{Unified weak equidistribution formulation}
\label{sec:equi-weak}

Let $d\in\{2,3\}$ and fix $s\in[0,T)$. By \eqref{eq:equi-jacobian}, $J^2\boldsymbol n=\overline J\boldsymbol N$. Testing \eqref{eq:equi-sd-local} with $\phi$ and $\boldsymbol\omega$ and integrating by parts gives the unified weak formulation: find $(\boldsymbol X_{s,t},\mathcal H)\in[H^1(\Gamma(s))]^d\times H^1(\Gamma(s))$, with $\boldsymbol X_{s,s}=\operatorname{id}_{\Gamma(s)}$, such that, for $t\in(s,T]$,
\begin{subequations}
\label{eq:equi-sd-weak}
\begin{align}
&\overline J(\boldsymbol N\cdot\partial_t\boldsymbol X_{s,t},\phi)_{\Gamma(s)}
=-\bigl(\boldsymbol C_{s,t}\nabla_{\Gamma(s)}\mathcal H,\nabla_{\Gamma(s)}\phi\bigr)_{\Gamma(s)},
&&\forall\phi\in H^1(\Gamma(s)), \label{eq:equi-sd-weak-normal}\\[2mm]
&\overline J(\mathcal H\boldsymbol N,\boldsymbol\omega)_{\Gamma(s)}
=\bigl\langle(\nabla_{\Gamma(s)}\boldsymbol X_{s,t})\boldsymbol C_{s,t},\nabla_{\Gamma(s)}\boldsymbol\omega\bigr\rangle_{\Gamma(s)},
&&\forall\boldsymbol\omega\in[H^1(\Gamma(s))]^d, \label{eq:equi-sd-weak-curvature}
\end{align}
\end{subequations}

\begin{theorem}[Preservation and dissipation]
\label{thm:equi-weak-structure}
A sufficiently smooth solution $(\boldsymbol X_{s,t},\mathcal H)$ of \eqref{eq:equi-sd-weak} satisfies
\begin{equation*}
|\Omega(t_2)|=|\Omega(t_1)|,\qquad |\Gamma(t_2)|\leq|\Gamma(t_1)|,\qquad s\leq t_1\leq t_2\leq T.
\end{equation*}
\end{theorem}

\begin{proof}
By the first identities in \eqref{eq:unified-volume}--\eqref{eq:unified-dissipation} and \eqref{eq:equi-sd-weak-normal} with $\phi=1$ and $\phi=\mathcal H$,
\begin{equation*}
\frac{\mathrm d}{\mathrm dt}|\Omega(t)|=0,
\qquad
\frac{\mathrm d}{\mathrm dt}|\Gamma(t)|
=-\overline J^{-1}\bigl(\boldsymbol C_{s,t}\nabla_{\Gamma(s)}\mathcal H,\nabla_{\Gamma(s)}\mathcal H\bigr)_{\Gamma(s)}\leq0,
\end{equation*}
since $\boldsymbol C_{s,t}$ is positive definite. Integrating over $[t_1,t_2]$ proves the result.
\end{proof}
\section{Structure-preserving collocation schemes}
\label{sec:direct-rk}

We discretize \eqref{eq:unified-sd-weak} and \eqref{eq:equi-sd-weak} in space by isoparametric finite elements and in time by Runge--Kutta (RK) collocation. On each time step, the reference hypersurface is the current discrete hypersurface, so that all stages are posed on one fixed mesh.

\subsection{Isoparametric geometry and Runge--Kutta collocation}
\label{sec:rk-geometry}

Let $t_m=m\tau$, $m=0,\ldots,M$, with $\tau=T/M$, and let $\Gamma^m\subset\mathbb R^d$ be a closed, consistently oriented isoparametric hypersurface enclosing $\Omega^m$. Following \cite[Sections~2.1 and~3.1]{GarckeNuernbergPraetoriusZhang2025}, let $\widehat K$ be the reference interval ($d=2$) or triangle ($d=3$), and write $\Gamma^m=\bigcup_{j=1}^{J}\sigma_j^m$ with $\sigma_j^m=\boldsymbol F_j^m(\widehat K)$ and regular maps $\boldsymbol F_j^m\in[\mathcal P_k(\widehat K)]^d$, $k\geq1$. Let
\begin{equation}\label{eq:high-order-fe-space}
\mathbb V^m_k := \left\{ v^h\in H^1(\Gamma^m): v^h\circ\boldsymbol F_j^m\in\mathcal P_k(\widehat K), \ j=1,\ldots,J \right\}.
\end{equation}
Surface differential operators and inner products are taken elementwise and assembled over $\Gamma^m$.

For an orientation-preserving, elementwise regular $\boldsymbol Y\in[\mathbb V_k^m]^d$, we define $\boldsymbol N^h(\boldsymbol Y)$ elementwise by \eqref{eq:normal-jacobian-tangent-formulas}, the discrete Jacobian $J^h(\boldsymbol Y):=|\boldsymbol N^h(\boldsymbol Y)|$, and the mean Jacobian
\begin{equation}\label{eq:equi-rk-area-ratio}
\overline J^h(\boldsymbol Y) :=\frac{|\boldsymbol Y(\Gamma^m)|}{|\Gamma^m|}= \frac{1}{|\Gamma^m|}\bigl(J^h(\boldsymbol Y),1\bigr)_{\Gamma^m}.
\end{equation}

On $[t_m,t_{m+1}]$, the local flow map from $\Gamma^m$ is approximated by a polynomial in time with values in $[\mathbb V_k^m]^d$. Let $\theta=(t-t_m)/\tau\in[0,1]$, choose collocation nodes $0<c_1<\cdots<c_s\leq1$ with Lagrange basis $\ell_i\in\mathcal P_{s-1}$, $\ell_i(c_j)=\delta_{ij}$, and let
\begin{equation*}
a_{ij}=\int_0^{c_i}\ell_j(\theta)\,\mathrm d\theta,\qquad b_i=\int_0^1\ell_i(\theta)\,\mathrm d\theta
\end{equation*}
be the associated RK coefficients. Given stage velocities $\delta\boldsymbol X^{m,i}\in[\mathbb V_k^m]^d$, the collocation polynomial is the unique $\boldsymbol Y$ of degree $s$ in $\theta$ with $\boldsymbol Y(0)=\boldsymbol X^m:=\operatorname{id}_{\Gamma^m}$ and $ \partial_\theta\boldsymbol Y(\theta)=\tau\sum_{i=1}^s\ell_i(\theta)\delta\boldsymbol X^{m,i}$
so that $\partial_\theta\boldsymbol Y(c_i)=\tau\delta\boldsymbol X^{m,i}$. Thus, \begin{equation}\label{eq:equi-rk-reconstruction}
\boldsymbol Y(\theta):=\boldsymbol X^m + \tau\sum_{j=1}^s\left(\int_0^\theta\ell_j(\xi)\,\mathrm d\xi\right)\delta\boldsymbol X^{m,j}.
\end{equation}
Its stage values and endpoint are the usual RK stages and update,
\begin{equation}\label{eq:stage-values-rk}
\boldsymbol X^{m,i}:=\boldsymbol Y(c_i)=\boldsymbol X^m+\tau\sum_{j=1}^s a_{ij}\delta\boldsymbol X^{m,j},
\quad
\boldsymbol X^{m+1}:=\boldsymbol Y(1)=\boldsymbol X^m+\tau\sum_{i=1}^s b_i\delta\boldsymbol X^{m,i},
\end{equation}
and $\Gamma^{m+1}:=\boldsymbol X^{m+1}(\Gamma^m)$ is the reference hypersurface for the next step. Throughout, we assume that each $\boldsymbol Y(\theta)$ is an orientation-preserving homeomorphism onto its image and regular on every element.

An RK method is \emph{algebraically stable} if $b_i\geq0$ and the symmetric matrix
\begin{equation}\label{eq:equi-rk-stability-matrix}
\mathsf M_{ij}:=b_i a_{ij}+b_j a_{ji}-b_i b_j,\qquad i,j=1,\ldots,s,
\end{equation}
is positive semidefinite. Gauss and Radau IIA collocation methods are algebraically stable with $b_i>0$ \cite[Section~IV.12]{HairerWanner1996}; in the computations, we use the implicit midpoint rule ($s=1$, $c_1=\frac12$) and the two-stage Radau IIA method ($c_1=\frac13$, $c_2=1$), of classical orders two and three.

\subsection{Structure-preserving discretization of the conformal formulation}
\label{sec:direct-rk-scheme}

In this subsection, we discretize \eqref{eq:unified-sd-weak} with an RK collocation method. Each stage equation is \eqref{eq:unified-sd-weak} at $\theta=c_i$, with $\partial_t\boldsymbol X_{s,t}$ replaced by $\delta\boldsymbol X^{m,i}$ and $\boldsymbol N$ replaced by the path average
\begin{equation}\label{eq:equi-rk-normal-average}
\overline{\boldsymbol N}^{m,i}:=\frac1{b_i}\int_0^1\ell_i(\theta)
\boldsymbol N^h(\boldsymbol Y(\theta))\,\mathrm d\theta.
\end{equation}
Since $\boldsymbol N^h(\boldsymbol Y(\theta))$ has degree $s$ ($d=2$) or $2s$ ($d=3$) in $\theta$, the integral is evaluated exactly by Gauss quadrature. With $\overline J^{m,i}:=\overline J^h(\boldsymbol X^{m,i})$, the scheme reads: find $(\delta\boldsymbol X^{m,i},\mathcal H^{m,i})\in[\mathbb V_k^m]^d\times\mathbb V_k^m$, $i=1,\ldots,s$, such that
\begin{subequations}\label{eq:direct-rk-scheme}
\begin{align}
&(\overline J^{m,i})^{3-d}(\overline{\boldsymbol N}^{m,i}\cdot\delta\boldsymbol X^{m,i},\phi^h)_{\Gamma^m}
=-(\nabla_{\Gamma^m}\mathcal H^{m,i},\nabla_{\Gamma^m}\phi^h)_{\Gamma^m},
&& \forall \phi^h \in \mathbb V_k^m,\label{eq:direct-rk-normal}\\[2mm]
&(\overline J^{m,i})^{3-d}(\mathcal H^{m,i}\overline{\boldsymbol N}^{m,i},\boldsymbol\omega^h)_{\Gamma^m}
=\langle\nabla_{\Gamma^m}\boldsymbol X^{m,i},\nabla_{\Gamma^m}\boldsymbol\omega^h\rangle_{\Gamma^m},
&& \forall \boldsymbol\omega^h \in [\mathbb V_k^m]^d,\label{eq:direct-rk-curvature}
\end{align}
\end{subequations}
together with \eqref{eq:stage-values-rk}. For $d=3$, the factor $(\overline J^{m,i})^{3-d}$ equals one.

\begin{remark}[Midpoint rule]
\label{rem:linear-rk-polynomial}
For the midpoint rule, $\boldsymbol Y$ is linear in $\theta$, $\boldsymbol X^{m,1}=\frac12(\boldsymbol X^m+\boldsymbol X^{m+1})$, and \eqref{eq:equi-rk-normal-average} becomes
\begin{equation}\label{eq:midpoint-simpson-normal}
\overline{\boldsymbol N}^{m,1} =\begin{cases}
\boldsymbol N^h(\boldsymbol X^{m,1}),&d=2,\\[1mm]
\frac16\bigl(\boldsymbol N^h(\boldsymbol X^m) +4\boldsymbol N^h(\boldsymbol X^{m,1}) +\boldsymbol N^h(\boldsymbol X^{m+1})\bigr),&d=3.
\end{cases}
\end{equation}
For $k=1$, this is the time-weighted normal of \cite{BaoZhao2021}.
\end{remark}

\begin{theorem}[Volume preservation and energy dissipation]
\label{thm:direct-rk-structure}
Let $d\in\{2,3\}$, let $b_i>0$, and assume exact integration. Then the scheme \eqref{eq:direct-rk-scheme} satisfies
\begin{equation}\label{eq:direct-rk-volume}
|\Omega^{m+1}|=|\Omega^m|.
\end{equation}
If, in addition, the RK method is algebraically stable, then
\begin{equation}\label{eq:direct-rk-energy-bound}
|\Gamma^{m+1}|+\tau\sum_{i=1}^s b_i\|\nabla_{\Gamma^m}\mathcal H^{m,i}\|_{L^2(\Gamma^m)}^2\leq|\Gamma^m|.
\end{equation}
\end{theorem}

\begin{proof}
\emph{Volume preservation.} The transport identity for $\Gamma(\theta):=\boldsymbol Y(\theta)(\Gamma^m)$ \cite{BarrettGarckeNuernberg2020}, together with \eqref{eq:equi-rk-reconstruction} and \eqref{eq:equi-rk-normal-average}, gives
\begin{equation}\label{eq:equi-rk-volume-increment}
|\Omega^{m+1}|-|\Omega^m|
=\int_0^1\bigl(\boldsymbol N^h(\boldsymbol Y(\theta))\cdot\partial_\theta\boldsymbol Y(\theta),1\bigr)_{\Gamma^m}\mathrm d\theta
=\tau\sum_{i=1}^s b_i\bigl(\overline{\boldsymbol N}^{m,i}\cdot\delta\boldsymbol X^{m,i},1\bigr)_{\Gamma^m}.
\end{equation}
Taking $\phi^h=1$ in \eqref{eq:direct-rk-normal} and using $\overline J^{m,i}>0$ shows that every term in the last sum vanishes. Hence \eqref{eq:direct-rk-volume} follows.

\emph{Energy dissipation.} For $d=2$, the elementary inequality $2a\leq a^2+1$ gives
$J^h(\boldsymbol X^{m+1})=|\nabla_{\Gamma^m}\boldsymbol X^{m+1}|
\leq\frac12(|\nabla_{\Gamma^m}\boldsymbol X^{m+1}|^2+1)$, while
$|\nabla_{\Gamma^m}\boldsymbol X^m|^2=1$. For $d=3$, \cite[Lemma~55]{BarrettGarckeNuernberg2020} gives
$J^h(\boldsymbol X^{m+1})\leq\frac12|\nabla_{\Gamma^m}\boldsymbol X^{m+1}|^2$, while
$|\nabla_{\Gamma^m}\boldsymbol X^m|^2=2$. Thus, in either case, integration over $\Gamma^m$ yields
\begin{align}
|\Gamma^{m+1}|-|\Gamma^m|
&\leq\frac12\left(\|\nabla_{\Gamma^m}\boldsymbol X^{m+1}\|_{L^2(\Gamma^m)}^2
-\|\nabla_{\Gamma^m}\boldsymbol X^m\|_{L^2(\Gamma^m)}^2\right)
\notag\\
&=\tau\sum_{i=1}^s b_i\langle\nabla_{\Gamma^m}\boldsymbol X^{m,i},\nabla_{\Gamma^m}\delta\boldsymbol X^{m,i}\rangle_{\Gamma^m}\notag\\
&-\frac{\tau^2}{2}\sum_{i,j=1}^s\mathsf M_{ij}\langle\nabla_{\Gamma^m}\delta\boldsymbol X^{m,i},\nabla_{\Gamma^m}\delta\boldsymbol X^{m,j}\rangle_{\Gamma^m}.
\label{eq:direct-rk-quadratic-identity}
\end{align}
The equality follows from the stage relations \eqref{eq:stage-values-rk} and the RK quadratic identity \cite[Section~IV.12]{HairerWanner1996}.
The double sum is nonnegative because $\mathsf M$ is positive semidefinite. Testing \eqref{eq:direct-rk-normal} with $\phi^h=\mathcal H^{m,i}$ and \eqref{eq:direct-rk-curvature} with $\boldsymbol\omega^h=\delta\boldsymbol X^{m,i}$ gives
\begin{equation}\label{eq:direct-rk-tested}
\langle\nabla_{\Gamma^m}\boldsymbol X^{m,i},\nabla_{\Gamma^m}\delta\boldsymbol X^{m,i}\rangle_{\Gamma^m}
=-\|\nabla_{\Gamma^m}\mathcal H^{m,i}\|_{L^2(\Gamma^m)}^2.
\end{equation}
Combining \eqref{eq:direct-rk-quadratic-identity} and \eqref{eq:direct-rk-tested} gives
\begin{equation*}
|\Gamma^{m+1}|-|\Gamma^m|
\leq-\tau\sum_{i=1}^s b_i\|\nabla_{\Gamma^m}\mathcal H^{m,i}\|_{L^2(\Gamma^m)}^2,
\end{equation*}
which is \eqref{eq:direct-rk-energy-bound}.
\end{proof}

\begin{remark}[Spatial quadrature]
\label{rem:direct-spatial-quadrature}
Exact volume preservation is retained provided that the quadrature rule integrates
$(\overline{\boldsymbol N}^{m,i}\cdot\delta\boldsymbol X^{m,i},1)_{\Gamma^m}$
exactly. Algebraic degrees of exactness $2k-1$ for $d=2$ and $3k-2$ for
$d=3$ are sufficient \cite{GarckeNuernbergPraetoriusZhang2025}. In our
computations, we use quadrature rules of algebraic degrees $3k-1$ and $4k-2$,
respectively, which ensure exact integration of all terms involving
$\overline{\boldsymbol N}^{m,i}$ in \eqref{eq:direct-rk-scheme}. For energy
dissipation, it suffices to use any positive-weight quadrature rule consistently,
with $|\Gamma^m|$ replaced by its corresponding quadrature value, since
\eqref{eq:direct-rk-quadratic-identity} remains valid in quadrature form.
\end{remark}
\begin{remark}[General Runge--Kutta methods]
\label{rem:general-rk}
The construction is not restricted to collocation methods. 
Let $(A,\boldsymbol b,\boldsymbol c)$ be an $s$-stage Runge--Kutta method with
$b_i>0$ that admits a $C^1$ reconstruction
$\boldsymbol Y(\theta) = \boldsymbol X^m +\tau\sum_{i=1}^s \alpha_i(\theta)\delta\boldsymbol X^{m,i},$ satisfying
$ \alpha_i(0)=0,\, \alpha_i(c_j)=a_{ji},\,
\alpha_i(1)=b_i$.
Defining
$$
\overline{\boldsymbol N}^{m,i}
:=
\frac1{b_i}\int_0^1
\alpha_i'(\theta)\boldsymbol N^h(\boldsymbol Y(\theta))\,\mathrm d\theta
$$
gives the same transport identity and hence exact volume preservation. If, in addition, the Runge--Kutta method is algebraically stable, the same argument also yields energy dissipation. The collocation construction is the special case
$\alpha_i(\theta)=\int_0^\theta\ell_i(\xi)\,\mathrm d\xi$.
\end{remark}
 \subsection{Structure-preserving discretization of the equidistribution formulation}
\label{sec:equi-rk}

We discretize \eqref{eq:equi-sd-weak} with the spaces, collocation polynomial and path averaged normal of \cref{sec:rk-geometry}. For an orientation-preserving, elementwise regular $\boldsymbol Y\in[\mathbb V_k^m]^d$, we further define elementwise the discrete counterparts of \eqref{eq:relative-metric} and \eqref{eq:equi-C},
\begin{equation}
\label{eq:equi-rk-GC}
\begin{aligned}
\boldsymbol G^h(\boldsymbol Y)&:=(\nabla_{\Gamma^m}\boldsymbol Y)^T\nabla_{\Gamma^m}\boldsymbol Y+\boldsymbol n^m\otimes\boldsymbol n^m,\\
\boldsymbol C^h(\boldsymbol Y)&:=\operatorname{cof}\boldsymbol G^h(\boldsymbol Y)=[J^h(\boldsymbol Y)]^2[\boldsymbol G^h(\boldsymbol Y)]^{-1},
\end{aligned}
\end{equation}
where $\boldsymbol n^m$ is the elementwise unit normal of $\Gamma^m$ and $[J^h(\boldsymbol Y)]^2=\det\boldsymbol G^h(\boldsymbol Y)$. Both matrices are symmetric positive definite. In place of $\nabla_{\Gamma^m}\boldsymbol X^{m,i}$, the curvature equation uses the path-averaged geometric discrete gradient
\begin{equation}
\label{eq:equi-rk-T}
\overline{\boldsymbol T}^{m,i}:=\frac1{b_i}\int_0^1\ell_i(\theta)\,(\nabla_{\Gamma^m}\boldsymbol Y(\theta))\,\boldsymbol C^h(\boldsymbol Y(\theta))\,\mathrm d\theta.
\end{equation}
With $\overline J^{m,i}:=\overline J^h(\boldsymbol X^{m,i})$ and $\boldsymbol C^{m,i}:=\boldsymbol C^h(\boldsymbol X^{m,i})$, the scheme reads: for $i=1,\ldots,s$, find $(\delta\boldsymbol X^{m,i},\mathcal H^{m,i})\in[\mathbb V_k^m]^d\times\mathbb V_k^m$, such that
\begin{subequations}
\label{eq:equi-rk-scheme}
\begin{align}
&\overline J^{m,i}(\overline{\boldsymbol N}^{m,i}\cdot\delta\boldsymbol X^{m,i},\phi^h)_{\Gamma^m}
=-(\boldsymbol C^{m,i}\nabla_{\Gamma^m}\mathcal H^{m,i},\nabla_{\Gamma^m}\phi^h)_{\Gamma^m},
&&\forall\phi^h\in\mathbb V_k^m,\label{eq:equi-rk-normal}\\[2mm]
&\overline J^{m,i}(\mathcal H^{m,i}\overline{\boldsymbol N}^{m,i},\boldsymbol\omega^h)_{\Gamma^m}
=\langle\overline{\boldsymbol T}^{m,i},\nabla_{\Gamma^m}\boldsymbol\omega^h\rangle_{\Gamma^m},
&&\forall\boldsymbol\omega^h\in[\mathbb V_k^m]^d,\label{eq:equi-rk-curvature}
\end{align}
\end{subequations}
together with \eqref{eq:stage-values-rk}.

Since $(\nabla_{\Gamma^m}\boldsymbol Y)\boldsymbol n^m=\boldsymbol0$, only the tangential block of $\boldsymbol C^h(\boldsymbol Y)$ enters \eqref{eq:equi-rk-T}; it has degree $0$ ($d=2$) or $2s$ ($d=3$) in $\theta$. Hence the integrand in \eqref{eq:equi-rk-T} has degree $2s-1$ or $4s-1$, and $s$ or $2s$ Gauss--Legendre points evaluate it exactly.

\begin{remark}[Equivalence in two dimensions with Gauss collocation]
For $d=2$, $\boldsymbol C^h(\boldsymbol Y)$ is the identity on tangent
vectors, and hence
$\boldsymbol C^{m,i}\nabla_{\Gamma^m}\mathcal H^{m,i}
=\nabla_{\Gamma^m}\mathcal H^{m,i}$. For Gauss collocation,
$\overline{\boldsymbol T}^{m,i}
=\nabla_{\Gamma^m}\boldsymbol X^{m,i}$, so
\eqref{eq:equi-rk-scheme} coincides with \eqref{eq:direct-rk-scheme}.
\end{remark}

\begin{theorem}[Volume preservation and energy dissipation]
\label{thm:equi-rk-structure}
Let $d\in\{2,3\}$, let $b_i>0$, and assume exact integration. Then the scheme \eqref{eq:equi-rk-scheme} satisfies $|\Omega^{m+1}|=|\Omega^m|$ and
\begin{equation}
\label{eq:equi-rk-energy}
|\Gamma^{m+1}|+\tau\sum_{i=1}^s b_i(\boldsymbol C^{m,i}\nabla_{\Gamma^m}\mathcal H^{m,i},\nabla_{\Gamma^m}\mathcal H^{m,i})_{\Gamma^m}\leq|\Gamma^m|.
\end{equation}
\end{theorem}

\begin{proof}
\emph{Volume preservation.} Taking $\phi^h=1$ in \eqref{eq:equi-rk-normal}, we can derive $(\overline{\boldsymbol N}^{m,i}\cdot\,\delta\boldsymbol X^{m,i},\,1)_{\Gamma^m}\,=0$, and \eqref{eq:equi-rk-volume-increment} applies.

\emph{Energy dissipation.} Since $a-1\leq\frac12(a^2-1)$ and $\det \boldsymbol{G}^h(\boldsymbol{X}^{m}) = 1$,
\begin{align}
|\Gamma^{m+1}|-|\Gamma^m| &=\int_{\Gamma^m}\bigl[J^h(\boldsymbol X^{m+1})-1\bigr]\,\mathrm d\Gamma^m
\leq\frac12\int_{\Gamma^m}\bigl([J^h(\boldsymbol X^{m+1})]^2-1\bigr)\,\mathrm d\Gamma^m\nonumber\\
&= \frac12\int_{\Gamma^m}\bigl(\det\boldsymbol G^h(\boldsymbol X^{m+1})-\det\boldsymbol G^h(\boldsymbol X^m)\bigr)\,\mathrm d\Gamma^m.
\end{align}
By Jacobi's formula \cite[Lemma~5.3]{EckGarckeKnabner2017}, 
\begin{equation}\label{eq:jacobi-formula}
\partial_\theta\det\boldsymbol G^h(\boldsymbol Y)=\boldsymbol C^h(\boldsymbol Y):\partial_\theta\boldsymbol G^h(\boldsymbol Y)=2(\nabla_{\Gamma^m}\boldsymbol Y)\boldsymbol C^h(\boldsymbol Y):\nabla_{\Gamma^m}\partial_\theta\boldsymbol Y,
\end{equation}
where the last equality uses the symmetry of $\boldsymbol C^h$. Integrating in $\theta$ and inserting \eqref{eq:equi-rk-reconstruction} and \eqref{eq:equi-rk-T},
\begin{equation}
\label{eq:equi-rk-area-increment}
|\Gamma^{m+1}|-|\Gamma^m|\leq\int_0^1\bigl\langle(\nabla_{\Gamma^m}\boldsymbol Y)\boldsymbol C^h(\boldsymbol Y),\nabla_{\Gamma^m}\partial_\theta\boldsymbol Y\bigr\rangle_{\Gamma^m}\mathrm d\theta
=\tau\sum_{i=1}^s b_i\langle\overline{\boldsymbol T}^{m,i},\nabla_{\Gamma^m}\delta\boldsymbol X^{m,i}\rangle_{\Gamma^m}.
\end{equation}
Taking $\phi^h=\mathcal H^{m,i}$ in \eqref{eq:equi-rk-normal} and $\boldsymbol\omega^h=\delta\boldsymbol X^{m,i}$ in \eqref{eq:equi-rk-curvature} gives $\langle\overline{\boldsymbol T}^{m,i},\nabla_{\Gamma^m}\delta\boldsymbol X^{m,i}\rangle_{\Gamma^m}=-(\boldsymbol C^{m,i}\nabla_{\Gamma^m}\mathcal H^{m,i},\nabla_{\Gamma^m}\mathcal H^{m,i})_{\Gamma^m}$, which proves \eqref{eq:equi-rk-energy}.
\end{proof}

In contrast to \cref{thm:direct-rk-structure}, algebraic stability is not required: the path average \eqref{eq:equi-rk-T} reproduces the exact increment of $\det\boldsymbol G^h$ along the collocation polynomial. The extension to the general Runge--Kutta methods of Remark~\ref{rem:general-rk} follows by replacing $\ell_i$ with $\alpha_i'$ in \eqref{eq:equi-rk-T}; the same energy estimate remains valid without algebraic stability. The spatial quadrature statements of Remark~\ref{rem:direct-spatial-quadrature} also carry over, since \eqref{eq:jacobi-formula} holds pointwise.

\subsection{Weakly nonlinear schemes with predicted coefficients}
\label{sec:direct-rk-predictor}
 
We reduce the nonlinearity of \eqref{eq:direct-rk-scheme} and
\eqref{eq:equi-rk-scheme} by freezing selected coefficients at predicted
stage surfaces. Let $\widetilde{\boldsymbol X}^{m,i}\in[\mathbb V_k^m]^d$,
$i=1,\ldots,s$, be regular, orientation-preserving predictions of
$\boldsymbol X^{m,i}$ obtained from a lower-order method or by extrapolation,
and set $\widetilde J^{m,i}:=\overline J^h(\widetilde{\boldsymbol X}^{m,i})>0$
and $\widetilde{\boldsymbol C}^{m,i}:=\boldsymbol C^h(\widetilde{\boldsymbol X}^{m,i})$.

The predicted-coefficient variant of \eqref{eq:direct-rk-scheme} is: find $(\delta\boldsymbol X^{m,i},\mathcal H^{m,i})\in[\mathbb V_k^m]^d\times\mathbb V_k^m$, $i=1,\ldots,s$, such that
\begin{subequations}\label{eq:direct-rk-predicted-scheme}
\begin{align}
&(\widetilde J^{m,i})^{3-d}(\overline{\boldsymbol N}^{m,i}\cdot\delta\boldsymbol X^{m,i},\phi^h)_{\Gamma^m} =-(\nabla_{\Gamma^m}\mathcal H^{m,i},\nabla_{\Gamma^m}\phi^h)_{\Gamma^m},&& \forall \phi^h \in \mathbb V_k^m,\label{eq:direct-rk-predicted-normal}\\[2mm]
&(\widetilde J^{m,i})^{3-d}(\mathcal H^{m,i}\overline{\boldsymbol N}^{m,i},\boldsymbol\omega^h)_{\Gamma^m} =\langle\nabla_{\Gamma^m}\boldsymbol X^{m,i},\nabla_{\Gamma^m}\boldsymbol\omega^h\rangle_{\Gamma^m},&& \forall \boldsymbol\omega^h \in [\mathbb V_k^m]^d, \label{eq:direct-rk-predicted-curvature}
\end{align}
\end{subequations}
together with \eqref{eq:stage-values-rk}. Only the Jacobian coefficient is frozen, while the path average $\overline{\boldsymbol N}^{m,i}$ is retained; for $d=3$, \eqref{eq:direct-rk-predicted-scheme} coincides with \eqref{eq:direct-rk-scheme}.

\begin{theorem}[Structure preservation of predicted conformal scheme]
\label{thm:predicted-conformal}
Under the assumptions of \cref{thm:direct-rk-structure}, the scheme \eqref{eq:direct-rk-predicted-scheme} satisfies $|\Omega^{m+1}|=|\Omega^m|$ and
\begin{equation}
\label{eq:direct-rk-predicted-energy}
|\Gamma^{m+1}|+
\tau\sum_{i=1}^s b_i
\|\nabla_{\Gamma^m}\mathcal H^{m,i}\|_{L^2(\Gamma^m)}^2
\leq |\Gamma^m|.
\end{equation}
\end{theorem}
 
\begin{proof}
Since $\widetilde J^{m,i}>0$, taking $\phi^h=1$ in \eqref{eq:direct-rk-predicted-normal} gives $(\overline{\boldsymbol N}^{m,i}\cdot\delta\boldsymbol X^{m,i},1)_{\Gamma^m}=0$. Volume preservation then follows directly from \eqref{eq:equi-rk-volume-increment}. Taking $\phi^h=\mathcal H^{m,i}$ in \eqref{eq:direct-rk-predicted-normal} and $\boldsymbol\omega^h=\delta\boldsymbol X^{m,i}$ in \eqref{eq:direct-rk-predicted-curvature}, the common factor $(\widetilde J^{m,i})^{3-d}$ cancels and gives
\begin{equation*}
\langle\nabla_{\Gamma^m}\boldsymbol X^{m,i},
\nabla_{\Gamma^m}\delta\boldsymbol X^{m,i}\rangle_{\Gamma^m}
=-\|\nabla_{\Gamma^m}\mathcal H^{m,i}\|_{L^2(\Gamma^m)}^2.
\end{equation*}
The energy estimate follows from \eqref{eq:direct-rk-quadratic-identity} and the positive semidefiniteness of the algebraic-stability matrix $\mathsf M$.
\end{proof}
 
\begin{remark}[Linearized scheme]
\label{rem:sd-linearized-variant}
Replacing, in addition, $\overline{\boldsymbol N}^{m,i}$ by $\widetilde{\boldsymbol N}^{m,i}:=\boldsymbol N^h(\widetilde{\boldsymbol X}^{m,i})$ in both equations of \eqref{eq:direct-rk-predicted-scheme} yields a linear stage system. Since the same vector appears in both equations, \eqref{eq:direct-rk-tested} and hence \eqref{eq:direct-rk-energy-bound} still hold, whereas exact volume preservation is generally lost.
\end{remark}

The predicted-coefficient variant of \eqref{eq:equi-rk-scheme} is: find $(\delta\boldsymbol X^{m,i},\mathcal H^{m,i})\in[\mathbb V_k^m]^d\times\mathbb V_k^m$, $i=1,\ldots,s$, such that
\begin{subequations}\label{eq:equi-rk-predicted-scheme}
\begin{align}
&\widetilde J^{m,i}(\overline{\boldsymbol N}^{m,i}\cdot\delta\boldsymbol X^{m,i},\phi^h)_{\Gamma^m} =-(\widetilde{\boldsymbol{C}}^{m,i}\nabla_{\Gamma^m}\mathcal H^{m,i},\nabla_{\Gamma^m}\phi^h)_{\Gamma^m},&& \forall \phi^h \in \mathbb V_k^m,\label{eq:equi-rk-predicted-normal}\\[2mm]
&\widetilde J^{m,i}(\mathcal H^{m,i}\overline{\boldsymbol N}^{m,i},\boldsymbol\omega^h)_{\Gamma^m} =\langle\overline{\boldsymbol{T}}^{m,i},\nabla_{\Gamma^m}\boldsymbol\omega^h\rangle_{\Gamma^m},&& \forall \boldsymbol\omega^h \in [\mathbb V_k^m]^d, \label{eq:equi-rk-predicted-curvature}
\end{align}
\end{subequations}
together with \eqref{eq:stage-values-rk}. The Jacobian coefficient and
cofactor matrix are frozen, so the remaining nonlinearity is confined to the
path averages $\overline{\boldsymbol N}^{m,i}$ and
$\overline{\boldsymbol T}^{m,i}$.

\begin{theorem}[Structure preservation of predicted equidistribution scheme]
\label{thm:predicted-equi}
Under the assumptions of \cref{thm:equi-rk-structure}, the scheme \eqref{eq:equi-rk-predicted-scheme} satisfies $|\Omega^{m+1}|=|\Omega^m|$ and
\begin{equation}
\label{eq:equi-rk-predicted-energy}
|\Gamma^{m+1}|+
\tau\sum_{i=1}^s b_i
(\widetilde{\boldsymbol C}^{m,i}\nabla_{\Gamma^m}\mathcal H^{m,i},
\nabla_{\Gamma^m}\mathcal H^{m,i})_{\Gamma^m}
\leq |\Gamma^m|.
\end{equation}
\end{theorem}

\begin{proof}
Since $\widetilde J^{m,i}>0$, taking $\phi^h=1$ in \eqref{eq:equi-rk-predicted-normal} gives $(\overline{\boldsymbol N}^{m,i}\cdot\delta\boldsymbol X^{m,i},1)_{\Gamma^m}=0$, and \eqref{eq:equi-rk-volume-increment} yields volume preservation. For the energy estimate, taking $\phi^h=\mathcal H^{m,i}$ in \eqref{eq:equi-rk-predicted-normal} and $\boldsymbol\omega^h=\delta\boldsymbol X^{m,i}$ in \eqref{eq:equi-rk-predicted-curvature} gives
\begin{equation*}
\langle\overline{\boldsymbol T}^{m,i},
\nabla_{\Gamma^m}\delta\boldsymbol X^{m,i}\rangle_{\Gamma^m}
=-(\widetilde{\boldsymbol C}^{m,i}\nabla_{\Gamma^m}\mathcal H^{m,i},
\nabla_{\Gamma^m}\mathcal H^{m,i})_{\Gamma^m}.
\end{equation*}
Combining \eqref{eq:equi-rk-area-increment} with the preceding identity
yields \eqref{eq:equi-rk-predicted-energy}.
\end{proof}

\paragraph{Stage predictors}
For the midpoint rule, $\widetilde{\boldsymbol X}^{m,1}$ can be obtained by one BGN-type step of size $\tau/2$ \cite{JiangSuZhangZhang2025}. For the two-stage Radau IIA method and $m\geq2$, we use quadratic extrapolation,
\begin{equation}\label{eq:radau-two-predictor}
\widetilde{\boldsymbol X}^{m,i}\circ\boldsymbol F_j^m=\frac{(c_i+1)(c_i+2)}{2}\boldsymbol F_j^m-c_i(c_i+2)\boldsymbol F_j^{m-1}+\frac{c_i(c_i+1)}{2}\boldsymbol F_j^{m-2}.
\end{equation}
For Radau IIA, the steps $m=0,1$ can be computed with the corresponding Midpoint method or fully
nonlinear scheme, \eqref{eq:direct-rk-scheme} or
\eqref{eq:equi-rk-scheme}, to initialize the extrapolation.

\section{Extension to volume-preserving mean curvature flow}
\label{sec:vpmcf-extension}

With the above parametrization, volume-preserving mean curvature flow (VP-MCF) is given by
\begin{equation}
\label{eq:vpmcf}
\begin{cases}
\boldsymbol n\cdot(\partial_t\boldsymbol X_t\circ\boldsymbol X_t^{-1})
=-(\mathcal H-\overline{\mathcal H}),\\[2mm]
\mathcal H\boldsymbol n
=-\Delta_{\Gamma(t)}\operatorname{id}_{\Gamma(t)},
\end{cases}
\qquad
\text{on }\Gamma(t),\quad t\in[0,T],
\end{equation}
where $\overline{\mathcal H}:=|\Gamma(t)|^{-1}\int_{\Gamma(t)}\mathcal H\,\mathrm d\Gamma(t)$.
Since only the velocity law differs from SD, the admissible parametrization
conditions and the equivalence results in
\cref{thm:conformal-equivalence,thm:equi-equivalence} remain unchanged, so
both formulations extend directly to VP-MCF.

\subsection{Weak formulations}

Fix $s\in[0,T)$. We use the geometric quantities defined in \cref{sec:unified-weak,sec:equi-weak}. Since $\mathrm d\Gamma(t)=J\,\mathrm d\Gamma(s)$ and $\boldsymbol N=J\boldsymbol n$, the velocity law pulls back to
$\boldsymbol N\cdot\partial_t\boldsymbol X_{s,t}=-J(\mathcal H-\overline{\mathcal H})$, where
\begin{equation}
\overline{\mathcal H}
:=\frac{(J\mathcal H,1)_{\Gamma(s)}}{(J,1)_{\Gamma(s)}}.
\label{eq:vpmcf-pullback-mean}
\end{equation}

\paragraph{Conformal formulation}
Here $\boldsymbol X_{s,t}$ is equidistributed for $d=2$ and conformal between
closed genus-zero surfaces for $d=3$. The Jacobian therefore satisfies
\begin{equation*}
J=\begin{cases}
\overline J,&d=2,\\[1mm]
\frac12|\nabla_{\Gamma(s)}\boldsymbol X_{s,t}|^2,&d=3,
\end{cases}
\end{equation*}
by \eqref{eq:equi-jacobian} and \eqref{eq:conformal-jacobian}. The weak conformal formulation is: find $(\boldsymbol X_{s,t},\mathcal H)\in[H^1(\Gamma(s))]^d\times H^1(\Gamma(s))$, with $\boldsymbol X_{s,s}=\operatorname{id}_{\Gamma(s)}$, such that, for $t\in(s,T]$,
\begin{subequations}
\label{eq:conformal-vpmcf-weak}
\begin{align}
&(\boldsymbol N\cdot\partial_t\boldsymbol X_{s,t},\phi)_{\Gamma(s)}
=-\bigl(J(\mathcal H-\overline{\mathcal H}),\phi\bigr)_{\Gamma(s)},
&&\forall\phi\in H^1(\Gamma(s)),
\label{eq:conformal-vpmcf-weak-normal}\\[2mm]
&\overline J^{3-d}(\mathcal H\boldsymbol N,\boldsymbol\omega)_{\Gamma(s)}
=\langle\nabla_{\Gamma(s)}\boldsymbol X_{s,t},
\nabla_{\Gamma(s)}\boldsymbol\omega\rangle_{\Gamma(s)},
&&\forall\boldsymbol\omega\in[H^1(\Gamma(s))]^d.
\label{eq:conformal-vpmcf-weak-curvature}
\end{align}
\end{subequations}

\paragraph{Equidistribution formulation}
Here $J=\overline J$ on $\Gamma(s)$, and hence \eqref{eq:vpmcf-pullback-mean} reduces to $\overline{\mathcal H}=(\mathcal H,1)_{\Gamma(s)}/(1,1)_{\Gamma(s)}$. The weak equidistribution formulation is: find $(\boldsymbol X_{s,t},\mathcal H)\in[H^1(\Gamma(s))]^d\times H^1(\Gamma(s))$, with $\boldsymbol X_{s,s}=\operatorname{id}_{\Gamma(s)}$, such that, for $t\in(s,T]$,
\begin{subequations}
\label{eq:equi-vpmcf-weak}
\begin{align}
&(\boldsymbol N\cdot\partial_t\boldsymbol X_{s,t},\phi)_{\Gamma(s)}
=-\overline J(\mathcal H-\overline{\mathcal H},\phi)_{\Gamma(s)},
&&\forall\phi\in H^1(\Gamma(s)),
\label{eq:equi-vpmcf-weak-normal}\\[2mm]
&\overline J(\mathcal H\boldsymbol N,\boldsymbol\omega)_{\Gamma(s)}
=\bigl\langle(\nabla_{\Gamma(s)}\boldsymbol X_{s,t})\boldsymbol C_{s,t},
\nabla_{\Gamma(s)}\boldsymbol\omega\bigr\rangle_{\Gamma(s)},
&&\forall\boldsymbol\omega\in[H^1(\Gamma(s))]^d.
\label{eq:equi-vpmcf-weak-curvature}
\end{align}
\end{subequations}
 
\begin{theorem}[Preservation and dissipation]
\label{thm:vpmcf-continuous-structure}
A sufficiently smooth solution $(\boldsymbol X_{s,t},\mathcal H)$ of either \eqref{eq:conformal-vpmcf-weak} or \eqref{eq:equi-vpmcf-weak} satisfies
\begin{equation*}
|\Omega(t_2)|=|\Omega(t_1)|,
\qquad
|\Gamma(t_2)|\leq|\Gamma(t_1)|,
\qquad
s\leq t_1\leq t_2\leq T.
\end{equation*}
\end{theorem}
 
\begin{proof}
By definition of $\overline{\mathcal H}$,
$(J(\mathcal H-\overline{\mathcal H}),1)_{\Gamma(s)}=0$.
Testing the normal equation of either formulation with $\phi=1$ and
$\phi=\mathcal H$, and using the first equalities in
\eqref{eq:unified-volume}--\eqref{eq:unified-dissipation}, gives
\begin{equation*}
\frac{\mathrm d}{\mathrm dt}|\Omega(t)|=0,
\qquad
\frac{\mathrm d}{\mathrm dt}|\Gamma(t)|
=-\bigl(J(\mathcal H-\overline{\mathcal H}),
\mathcal H-\overline{\mathcal H}\bigr)_{\Gamma(s)}\leq0.
\end{equation*}
Here $J=\overline J$ for the equidistribution formulation. Integration over
$[t_1,t_2]$ proves the result.
\end{proof}
 
\subsection{Structure-preserving discretizations}

\paragraph{Discretization of the conformal formulation}
We use the setting of \cref{sec:rk-geometry} and the regular predictors
$\widetilde{\boldsymbol X}^{m,i}$ from \cref{sec:direct-rk-predictor}. Define
$\widetilde J^{m,i}:=\overline J^h(\widetilde{\boldsymbol X}^{m,i})>0$ and set
\begin{equation}
\widetilde{\mathcal J}^{m,i}:=
\begin{cases}
\widetilde J^{m,i},&d=2,\\[1mm]
\frac12|\nabla_{\Gamma^m}\widetilde{\boldsymbol X}^{m,i}|^2,&d=3,
\end{cases}
\qquad
\overline{\mathcal H}^{m,i}
:=\frac{(\widetilde{\mathcal J}^{m,i}\mathcal H^{m,i},1)_{\Gamma^m}}
{(\widetilde{\mathcal J}^{m,i},1)_{\Gamma^m}}.
\label{eq:predicted-vpmcf-jacobian}
\end{equation}
For regular predictors, $\widetilde{\mathcal J}^{m,i}>0$; for $d=2$,
$\overline{\mathcal H}^{m,i}$ is the plain mean of $\mathcal H^{m,i}$ over
$\Gamma^m$. Using the path average \eqref{eq:equi-rk-normal-average}, the
predicted-coefficient conformal scheme is: find
$(\delta\boldsymbol X^{m,i},\mathcal H^{m,i})\in
[\mathbb V_k^m]^d\times\mathbb V_k^m$, $i=1,\ldots,s$, such that
\begin{subequations}
\label{eq:conformal-rk-vpmcf-predicted-scheme}
\begin{align}
&(\overline{\boldsymbol N}^{m,i}\cdot\delta\boldsymbol X^{m,i},\phi^h)_{\Gamma^m}
=-\bigl(\widetilde{\mathcal J}^{m,i}
(\mathcal H^{m,i}-\overline{\mathcal H}^{m,i}),\phi^h\bigr)_{\Gamma^m},
&&\forall\phi^h\in\mathbb V_k^m,
\label{eq:conformal-rk-vpmcf-predicted-normal}\\[2mm]
&(\widetilde J^{m,i})^{3-d}
(\mathcal H^{m,i}\overline{\boldsymbol N}^{m,i},\boldsymbol\omega^h)_{\Gamma^m}
=\langle\nabla_{\Gamma^m}\boldsymbol X^{m,i},
\nabla_{\Gamma^m}\boldsymbol\omega^h\rangle_{\Gamma^m},
&&\forall\boldsymbol\omega^h\in[\mathbb V_k^m]^d.
\label{eq:conformal-rk-vpmcf-predicted-curvature}
\end{align}
\end{subequations}
together with \eqref{eq:stage-values-rk}.
 
\begin{theorem}[Structure preservation of predicted conformal scheme]
\label{thm:conformal-rk-vpmcf-predicted-structure}
Let $d\in\{2,3\}$, let $b_i>0$, and assume exact integration and regular
predictors with $\widetilde J^{m,i}>0$ and
$\widetilde{\mathcal J}^{m,i}>0$. Then
\eqref{eq:conformal-rk-vpmcf-predicted-scheme} satisfies
$|\Omega^{m+1}|=|\Omega^m|$. If, in addition, the Runge--Kutta method is
algebraically stable, then
\begin{equation}
\label{eq:conformal-rk-vpmcf-predicted-energy}
|\Gamma^{m+1}|+
\tau\sum_{i=1}^s b_i(\widetilde J^{m,i})^{3-d}
\bigl(\widetilde{\mathcal J}^{m,i}
(\mathcal H^{m,i}-\overline{\mathcal H}^{m,i}),
\mathcal H^{m,i}-\overline{\mathcal H}^{m,i}\bigr)_{\Gamma^m}
\leq|\Gamma^m|.
\end{equation}
\end{theorem}
 
\begin{proof}
By \eqref{eq:predicted-vpmcf-jacobian},
\begin{equation*}
\bigl(\widetilde{\mathcal J}^{m,i}(\mathcal H^{m,i}
-\overline{\mathcal H}^{m,i}),1\bigr)_{\Gamma^m}=0.
\end{equation*}
Taking $\phi^h=1$ in \eqref{eq:conformal-rk-vpmcf-predicted-normal} and using
\eqref{eq:equi-rk-volume-increment} gives volume preservation. Taking
$\phi^h=(\widetilde J^{m,i})^{3-d}\mathcal H^{m,i}$ in
\eqref{eq:conformal-rk-vpmcf-predicted-normal} and
$\boldsymbol\omega^h=\delta\boldsymbol X^{m,i}$ in
\eqref{eq:conformal-rk-vpmcf-predicted-curvature}, and using
\eqref{eq:predicted-vpmcf-jacobian}, gives
\begin{equation*}
\langle\nabla_{\Gamma^m}\boldsymbol X^{m,i},
\nabla_{\Gamma^m}\delta\boldsymbol X^{m,i}\rangle_{\Gamma^m}
=-(\widetilde J^{m,i})^{3-d}
\bigl(\widetilde{\mathcal J}^{m,i}
(\mathcal H^{m,i}-\overline{\mathcal H}^{m,i}),
\mathcal H^{m,i}-\overline{\mathcal H}^{m,i}\bigr)_{\Gamma^m}.
\end{equation*}
The estimate follows from \eqref{eq:direct-rk-quadratic-identity} and the positive semidefiniteness of $\mathsf M$.
\end{proof}

\begin{remark}[Linearized scheme]
\label{rem:vpmcf-linearized-conformal}
Replacing $\overline{\boldsymbol N}^{m,i}$ additionally by
$\widetilde{\boldsymbol N}^{m,i}$ in both equations of
\eqref{eq:conformal-rk-vpmcf-predicted-scheme} yields a linearized scheme
that retains energy dissipation, although exact volume preservation is
generally lost.
\end{remark}

\paragraph{Discretization of the equidistribution formulation}
Using the path averages \eqref{eq:equi-rk-normal-average} and
\eqref{eq:equi-rk-T} and the regular predictors of
\cref{sec:direct-rk-predictor}, we write
$\overline{\mathcal H}^{m,i}$ for the plain stage mean.
The predicted-coefficient scheme is: find $(\delta\boldsymbol X^{m,i},\mathcal H^{m,i})\in[\mathbb V_k^m]^d\times\mathbb V_k^m$, $i=1,\ldots,s$, such that
\begin{subequations}
\label{eq:equi-rk-vpmcf-predicted-scheme}
\begin{align}
&(\overline{\boldsymbol N}^{m,i}\cdot\delta\boldsymbol X^{m,i},\phi^h)_{\Gamma^m}
=-\widetilde J^{m,i}(\mathcal H^{m,i}-\overline{\mathcal H}^{m,i},\phi^h)_{\Gamma^m},
&&\forall\phi^h\in\mathbb V_k^m,
\label{eq:equi-rk-vpmcf-predicted-normal}\\[2mm]
&\widetilde J^{m,i}(\mathcal H^{m,i}\overline{\boldsymbol N}^{m,i},\boldsymbol\omega^h)_{\Gamma^m}
=\langle\overline{\boldsymbol T}^{m,i},\nabla_{\Gamma^m}\boldsymbol\omega^h\rangle_{\Gamma^m},
&&\forall\boldsymbol\omega^h\in[\mathbb V_k^m]^d,
\label{eq:equi-rk-vpmcf-predicted-curvature}
\end{align}
\end{subequations}
together with \eqref{eq:stage-values-rk}.

\begin{theorem}[Structure preservation of predicted equidistribution scheme]
\label{thm:equi-rk-vpmcf-structure}
Let $d\in\{2,3\}$, let $b_i>0$, and assume exact integration and regular predictors with $\widetilde J^{m,i}>0$. Then the scheme \eqref{eq:equi-rk-vpmcf-predicted-scheme} satisfies $|\Omega^{m+1}|=|\Omega^m|$ and
\begin{equation}
|\Gamma^{m+1}|+
\tau\sum_{i=1}^s b_i(\widetilde J^{m,i})^2
\|\mathcal H^{m,i}-\overline{\mathcal H}^{m,i}\|_{L^2(\Gamma^m)}^2
\leq|\Gamma^m|.
\label{eq:equi-rk-vpmcf-predicted-energy}
\end{equation}
\end{theorem}

\begin{proof}
Since $\overline{\mathcal H}^{m,i}$ is the stage mean,
\begin{equation*}
(\mathcal H^{m,i}-\overline{\mathcal H}^{m,i},1)_{\Gamma^m}=0.
\end{equation*}
Taking $\phi^h=1$ in \eqref{eq:equi-rk-vpmcf-predicted-normal} gives
$(\overline{\boldsymbol N}^{m,i}\cdot\delta\boldsymbol X^{m,i},1)_{\Gamma^m}=0$,
and \eqref{eq:equi-rk-volume-increment} yields volume preservation. Taking
$\phi^h=\mathcal H^{m,i}$ and
$\boldsymbol\omega^h=\delta\boldsymbol X^{m,i}$ gives
\begin{equation*}
\langle\overline{\boldsymbol T}^{m,i},\nabla_{\Gamma^m}\delta\boldsymbol X^{m,i}\rangle_{\Gamma^m}
=-(\widetilde J^{m,i})^2
\|\mathcal H^{m,i}-\overline{\mathcal H}^{m,i}\|_{L^2(\Gamma^m)}^2.
\end{equation*}
Here we used
$(\mathcal H^{m,i}-\overline{\mathcal H}^{m,i},
\mathcal H^{m,i})_{\Gamma^m}
=\|\mathcal H^{m,i}-\overline{\mathcal H}^{m,i}\|_{L^2(\Gamma^m)}^2$.
Combining \eqref{eq:equi-rk-area-increment} with the preceding identity
yields \eqref{eq:equi-rk-vpmcf-predicted-energy}. In particular,
algebraic stability is not required.
\end{proof}

The quadrature requirements of Remark~\ref{rem:direct-spatial-quadrature} apply to
both formulations: the normal--Jacobian coupling must be integrated exactly,
and the same positive-weight rule must be used in the stage equations and
the definitions of $\overline{\mathcal H}^{m,i}$. The conformal scheme additionally requires algebraic
stability for energy dissipation.

\section{Numerical experiments}
\label{sec:numerical-tests}

We focus primarily on the conformal formulation in the planar and surface
tests of Examples~1--4. Selected results for the equidistribution
formulation are reported separately in Examples~5--6.

\subsection{Computational setup and diagnostics}
\label{sec:numerical-setting}
All computations are performed in Firedrake with degree-$k$ isoparametric
finite elements. The implicit midpoint and two-stage Radau IIA methods are
denoted by Midpoint and Radau IIA. The stage predictors are those of
\cref{sec:direct-rk-predictor}, with fully nonlinear startup steps. Picard or
Newton iteration is stopped when the Euclidean norm of the algebraic update
is below $10^{-10}$ for curves and $10^{-8}$ for surfaces.

Following Remark \ref{rem:direct-spatial-quadrature}, the spatial quadrature has
degree of precision $3k-1$ for curves and $4k-2$ for surfaces. The path
averages are evaluated exactly in time. For exact algebraic solves, the
preservation defects vanish; computed defects are therefore expected at the
solver-tolerance level.

In our numerical tests, we use the following four initial geometries:
\begin{itemize}
\item \emph{Ellipse:} $(2\cos\alpha,\sin\alpha)$, with $0\leq\alpha<2\pi$.
\item \emph{Flower:} $\bigl((1+0.38\cos(5\alpha))\cos\alpha, (1+0.38\cos(5\alpha))\sin\alpha\bigr)$, with $0\leq\alpha<2\pi$.
\item \emph{Ellipsoid:} $\left(1.5\sin\vartheta\cos\varphi,
\sin\vartheta\sin\varphi,0.8\cos\vartheta\right)$, with
$0\leq\vartheta\leq\pi$ and $0\leq\varphi<2\pi$.
\item \emph{Cuboid:}
$\partial\left([-3,3]\times\left[-\frac12,\frac12\right]
\times\left[-\frac12,\frac12\right]\right)$.
\end{itemize}

For convergence tests, let $\Gamma_{h,\tau}(T)$ denote the numerical hypersurface obtained with spatial mesh size $h$ and time step $\tau$. We use the Hausdorff distance 
\begin{equation}\label{eq:hausdorff-distance}
d_H(\Gamma_1,\Gamma_2):=\max\left\{\sup_{\boldsymbol x\in\Gamma_1}\operatorname{dist}(\boldsymbol x,\Gamma_2),\sup_{\boldsymbol y\in\Gamma_2}\operatorname{dist}(\boldsymbol y,\Gamma_1)\right\},
\end{equation} 
and the successive refinement error is defined as
\begin{equation}\label{eq:numerical-error}
e_{h,\tau}(T):=d_H\bigl(\Gamma_{h,\tau}(T),\Gamma_{\frac{h}{2},\frac{\tau}{2}}(T)\bigr).
\end{equation}
To evaluate \eqref{eq:hausdorff-distance}, each high-order discrete
hypersurface is replaced by a sufficiently fine polyhedral sampling; the
observed orders are computed from consecutive errors.

To assess structure preservation and mesh quality, we record the curve length or surface area $|\Gamma^m|$, the defect in the enclosed area or volume $|\Omega^m|-|\Omega^0|$, and the mesh ratio
\begin{equation}\label{eq:numerical-mesh-ratio}
\mathcal R^m:=
\begin{cases}
\displaystyle\frac{\max_{1\leq j\leq J}|\sigma_j^m|}
{\min_{1\leq j\leq J}|\sigma_j^m|},& d=2,\\[3mm]
\displaystyle\max_{1\leq j\leq J}\frac{R_j^m}{r_j^m},& d=3.
\end{cases}
\end{equation}
Here $|\sigma_j^m|$ is the length of the $j$th curve element, while
$ R_j^m$ and $r_j^m$ are the circumradius and inradius of the triangle
formed by the vertices of the $j$th curved surface element. Smaller values
indicate better mesh quality.

\subsection{Planar tests with the conformal formulation}
\label{sec:numerical-curves}
For SD and VP-MCF, we use the predicted conformal schemes
\eqref{eq:direct-rk-predicted-scheme} and
\eqref{eq:conformal-rk-vpmcf-predicted-scheme}, respectively. 

\paragraph{Example 1: Convergence on an ellipse}
We assess convergence on the ellipse under SD and VP-MCF up to
$T=0.01$. We use Midpoint with spatial degree
$k=1$ and Radau IIA with $k=2$, and couple the refinements according to
$(h,\tau)=(\frac{h_0}{2^\ell},\frac{\tau_0}{2^\ell})$.

\begin{table}[htbp]
\centering
\caption{Hausdorff errors $e_{h,\tau}(T)$ and orders for the ellipse under SD at $T=0.01$, with $(h_0,\tau_0)=(0.125,0.005)$ and $\tau=0.04h$.}
\label{tab:curve-accuracy-sd}
\small
\begin{tabular}{ccccc}
\hline
 & \multicolumn{2}{c}{Midpoint ($k=1$)} & \multicolumn{2}{c}{Radau IIA ($k=2$)}\\
\cline{2-3}\cline{4-5}
$(h,\tau)$ & $e_{h,\tau}(T)$ & order & $e_{h,\tau}(T)$ & order\\
\hline
$(h_0,\tau_0)$ & $1.7716\times10^{-1}$ & -- & $1.9975\times10^{-2}$ & --\\
$(h_0/2,\tau_0/2)$ & $6.1305\times10^{-2}$ & 1.531 & $1.8886\times10^{-3}$ & 3.403\\
$(h_0/2^{2},\tau_0/2^{2})$ & $1.6461\times10^{-2}$ & 1.897 & $2.4904\times10^{-4}$ & 2.923\\
$(h_0/2^{3},\tau_0/2^{3})$ & $4.1980\times10^{-3}$ & 1.971 & $3.1952\times10^{-5}$ & 2.962\\
$(h_0/2^{4},\tau_0/2^{4})$ & $1.0542\times10^{-3}$ & 1.994 & $4.1196\times10^{-6}$ & 2.955\\
\hline
\end{tabular}
\end{table}

\begin{table}[htbp]
\centering
\caption{Hausdorff errors $e_{h,\tau}(T)$ and orders for the ellipse under VP-MCF at $T=0.01$, with $(h_0,\tau_0)=(0.125,0.005)$ and $\tau=0.04h$.}
\label{tab:curve-accuracy-vpmcf}
\small
\begin{tabular}{ccccc}
\hline
 & \multicolumn{2}{c}{Midpoint ($k=1$)} & \multicolumn{2}{c}{Radau IIA ($k=2$)}\\
\cline{2-3}\cline{4-5}
$(h,\tau)$ & $e_{h,\tau}(T)$ & order & $e_{h,\tau}(T)$ & order\\
\hline
$(h_0,\tau_0)$ & $1.6824\times10^{-1}$ & -- & $2.0478\times10^{-2}$ & --\\
$(h_0/2,\tau_0/2)$ & $6.5792\times10^{-2}$ & 1.355 & $3.5079\times10^{-3}$ & 2.545\\
$(h_0/2^{2},\tau_0/2^{2})$ & $2.0069\times10^{-2}$ & 1.713 & $3.6030\times10^{-4}$ & 3.283\\
$(h_0/2^{3},\tau_0/2^{3})$ & $5.3220\times10^{-3}$ & 1.915 & $5.3447\times10^{-5}$ & 2.753\\
$(h_0/2^{4},\tau_0/2^{4})$ & $1.3497\times10^{-3}$ & 1.979 & $7.2005\times10^{-6}$ & 2.892\\
\hline
\end{tabular}
\end{table}

Tables~\ref{tab:curve-accuracy-sd} and~
\ref{tab:curve-accuracy-vpmcf} show that Midpoint approaches second-order
convergence, while Radau IIA exhibits approximately third-order convergence
for both SD and VP-MCF. Under the coupled refinement, these rates indicate
that the predicted-coefficient treatment is consistent with the expected combined space--time accuracy.

\paragraph{Example 2: Evolution of a flower-shaped curve}
For the nonconvex five-lobed curve, Midpoint and Radau IIA are compared with the first-order Bao--Zhao-type
baseline for SD \cite{BaoZhao2021} and its volume-preserving counterpart for
VP-MCF. We take $h=2^{-7}$. For SD,
$T=0.03$ and $\tau=10^{-4}$; for VP-MCF, $T=0.3$ and $\tau=10^{-3}$. In
the Midpoint method with $k=1$, the mass terms are evaluated using the
standard mass-lumped inner product.

\begin{figure}[!htbp]
\centering
\includegraphics[width=\linewidth]{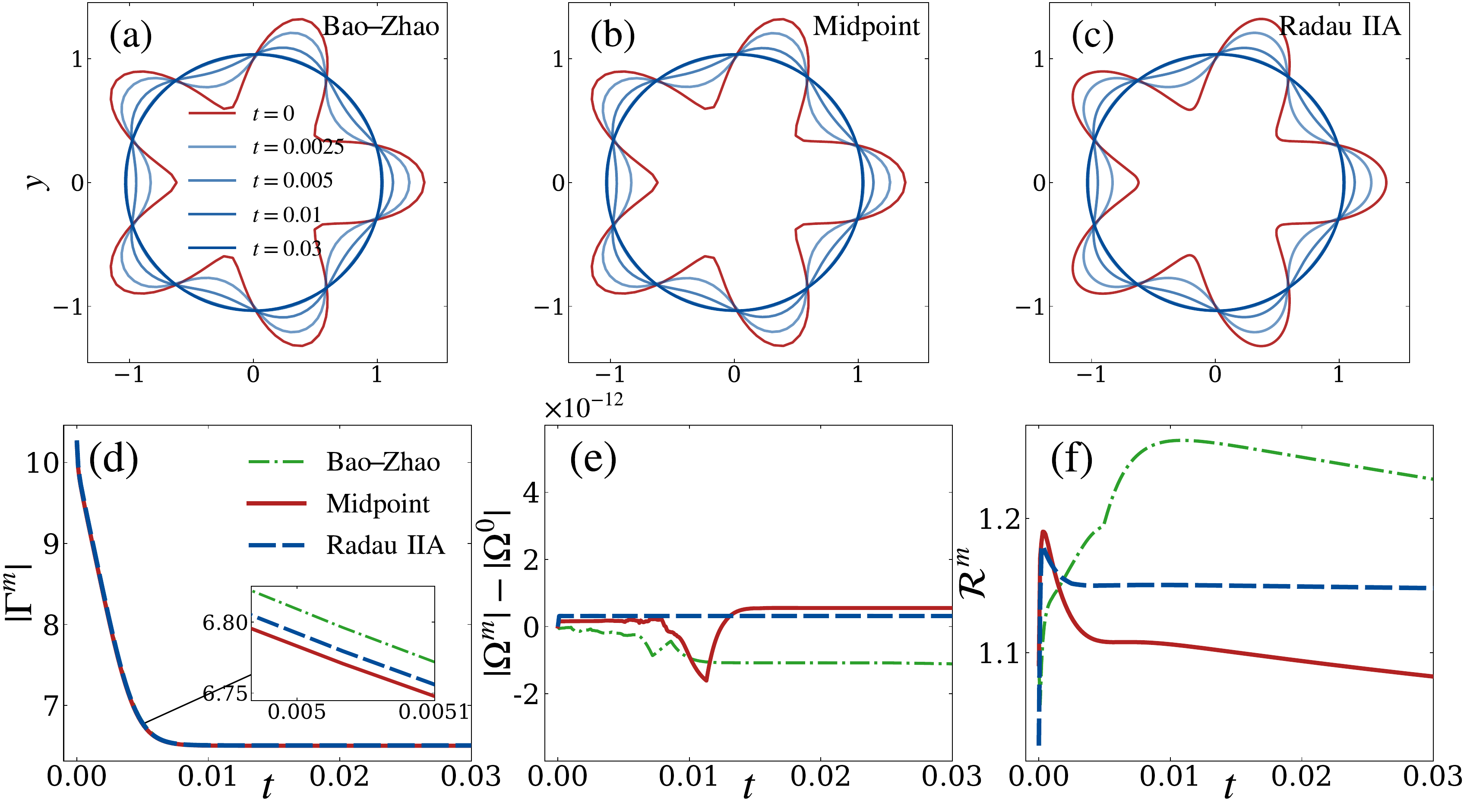}
\caption{SD evolution of a flower-shaped curve. (a)--(c): snapshots obtained with the Bao--Zhao-type baseline, Midpoint, and Radau IIA. The initial curve is red, and successive times are shown in increasingly dark blue. (d)--(f): curve length $|\Gamma^m|$, enclosed-area change $|\Omega^m|-|\Omega^0|$, and mesh ratio $R^m$.}
\label{fig:flower-sd}
\end{figure}

\begin{figure}[!htbp]
\centering

\includegraphics[width=\linewidth]{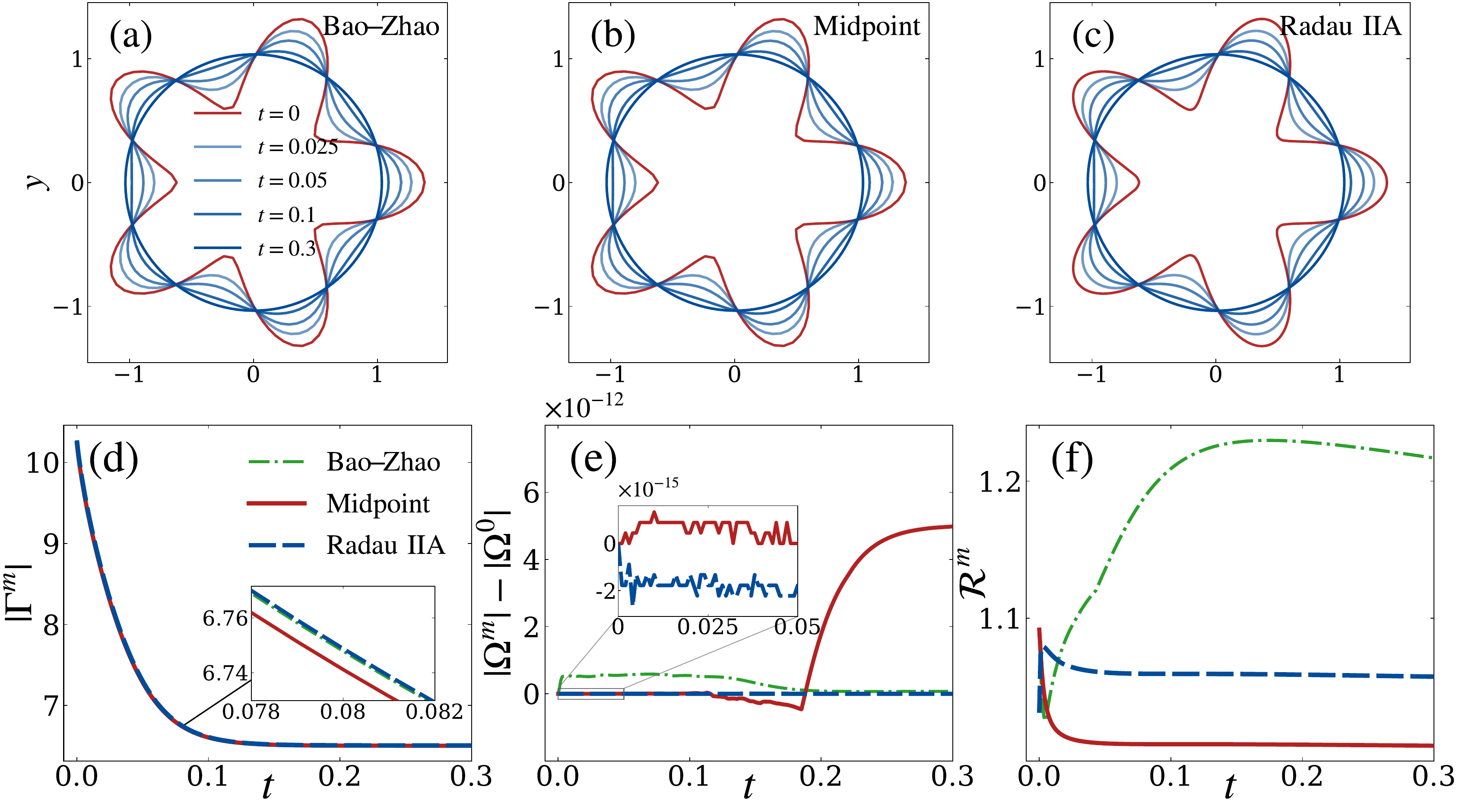}

\caption{VP-MCF evolution of a flower-shaped curve. (a)--(c): snapshots obtained with the Bao--Zhao-type baseline, Midpoint, and Radau IIA. The initial curve is red, and successive times are shown in increasingly dark blue. (d)--(f): curve length $|\Gamma^m|$, enclosed-area change $|\Omega^m|-|\Omega^0|$, and mesh ratio $R^m$.}
\label{fig:flower-vpmcf}
\end{figure}

Figures~\ref{fig:flower-sd} and~\ref{fig:flower-vpmcf} show that all three
methods smooth the lobes toward a circular shape, decrease the curve length,
and preserve the enclosed area to within the solver tolerance. Midpoint and
Radau IIA also maintain lower mesh
ratios than the Bao--Zhao-type baseline over most of the evolution.
Thus the higher-order schemes improve the tangential mesh distribution
without compromising the geometric evolution or the discrete structures.

\subsection{Surface tests with the conformal formulation}
\label{sec:numerical-surfaces}
We use the same predicted conformal schemes on closed surfaces in
$\mathbb R^3$.

\paragraph{Example 3: Convergence on an ellipsoid}
We assess convergence on the ellipsoid under SD and VP-MCF up to $T=0.005$.
We use Midpoint with spatial degree $k=1$ and Radau IIA with $k=2$, and
couple the refinements according to
$(h,\tau)=(\frac{h_0}{2^\ell},\frac{\tau_0}{2^\ell})$. We set
\begin{equation}\label{eq:surface-initial-mesh-size}
h_0:=\max_{1\leq j\leq J}\sqrt{|\sigma_j^0|},\qquad \tau_0=0.005,
\end{equation}
where $\{\sigma_j^0\}_{j=1}^{J}$ is the coarsest initial mesh for each spatial discretization.

\begin{table}[htbp]
\centering
\caption{Hausdorff errors $e_{h,\tau}(T)$ and orders for the ellipsoid under SD at $T=0.005$, with $h_0$ defined in \eqref{eq:surface-initial-mesh-size} and $\tau_0=0.005$.}
\label{tab:surface-accuracy-sd}
\small
\begin{tabular}{ccccc}
\hline
 & \multicolumn{2}{c}{Midpoint ($k=1$)} & \multicolumn{2}{c}{Radau IIA ($k=2$)}\\
\cline{2-3}\cline{4-5}
$(h,\tau)$ & $e_{h,\tau}(T)$ & order & $e_{h,\tau}(T)$ & order\\
\hline
$(h_0,\tau_0)$ & $2.2124\times10^{-1}$ & -- & $2.1876\times10^{-2}$ & --\\
$(h_0/2,\tau_0/2)$ & $6.6141\times10^{-2}$ & 1.742 & $3.6445\times10^{-3}$ & 2.586\\
$(h_0/2^{2},\tau_0/2^{2})$ & $1.7891\times10^{-2}$ & 1.886 & $5.3555\times10^{-4}$ & 2.767\\
$(h_0/2^{3},\tau_0/2^{3})$ & $4.7469\times10^{-3}$ & 1.914 & $8.2118\times10^{-5}$ & 2.705\\
\hline
\end{tabular}
\end{table}

\begin{table}[htbp]
\centering
\caption{Hausdorff errors $e_{h,\tau}(T)$ and orders for the ellipsoid under VP-MCF at $T=0.005$, with $h_0$ defined in \eqref{eq:surface-initial-mesh-size} and $\tau_0=0.005$.}
\label{tab:surface-accuracy-vpmcf}
\small
\begin{tabular}{ccccc}
\hline
 & \multicolumn{2}{c}{Midpoint ($k=1$)} & \multicolumn{2}{c}{Radau IIA ($k=2$)}\\
\cline{2-3}\cline{4-5}
$(h,\tau)$ & $e_{h,\tau}(T)$ & order & $e_{h,\tau}(T)$ & order\\
\hline
$(h_0,\tau_0)$ & $2.2920\times10^{-1}$ & -- & $2.4103\times10^{-2}$ & --\\
$(h_0/2,\tau_0/2)$ & $6.7679\times10^{-2}$ & 1.760 & $2.3474\times10^{-3}$ & 3.360\\
$(h_0/2^{2},\tau_0/2^{2})$ & $1.8065\times10^{-2}$ & 1.905 & $3.3333\times10^{-4}$ & 2.816\\
$(h_0/2^{3},\tau_0/2^{3})$ & $4.8220\times10^{-3}$ & 1.906 & $4.6591\times10^{-5}$ & 2.839\\
\hline
\end{tabular}
\end{table}

Tables~\ref{tab:surface-accuracy-sd} and~
\ref{tab:surface-accuracy-vpmcf} show that, under coupled space--time
refinement, Midpoint approaches second-order convergence and Radau IIA
exhibits a third-order trend for both flows, consistent with the
two-dimensional results. This agreement suggests that the stage prediction
does not visibly reduce the observed convergence order on surfaces.

\paragraph{Example 4: Evolution of a cuboid}
Although the initial $6\times1\times1$ cuboid lies outside the smooth setting
of the analysis, it provides a demanding test of rapid initial smoothing.
Both methods use quadratic elements ($k=2$), time step
$\tau=2\times10^{-4}$, and Newton iteration. We evolve to $T=0.6$ for SD
and $T=2$ for VP-MCF. Under SD, the ends become rounded and enlarge while the
middle remains narrower, producing a dumbbell-like intermediate shape.
Under VP-MCF, the cuboid becomes a rounded elongated surface without a
pronounced waist. In both cases, the initially sharp edges and corners are
smoothed, and the computed surfaces become progressively more rounded.

\begin{figure}[!htbp]
\centering
\includegraphics[width=\linewidth]{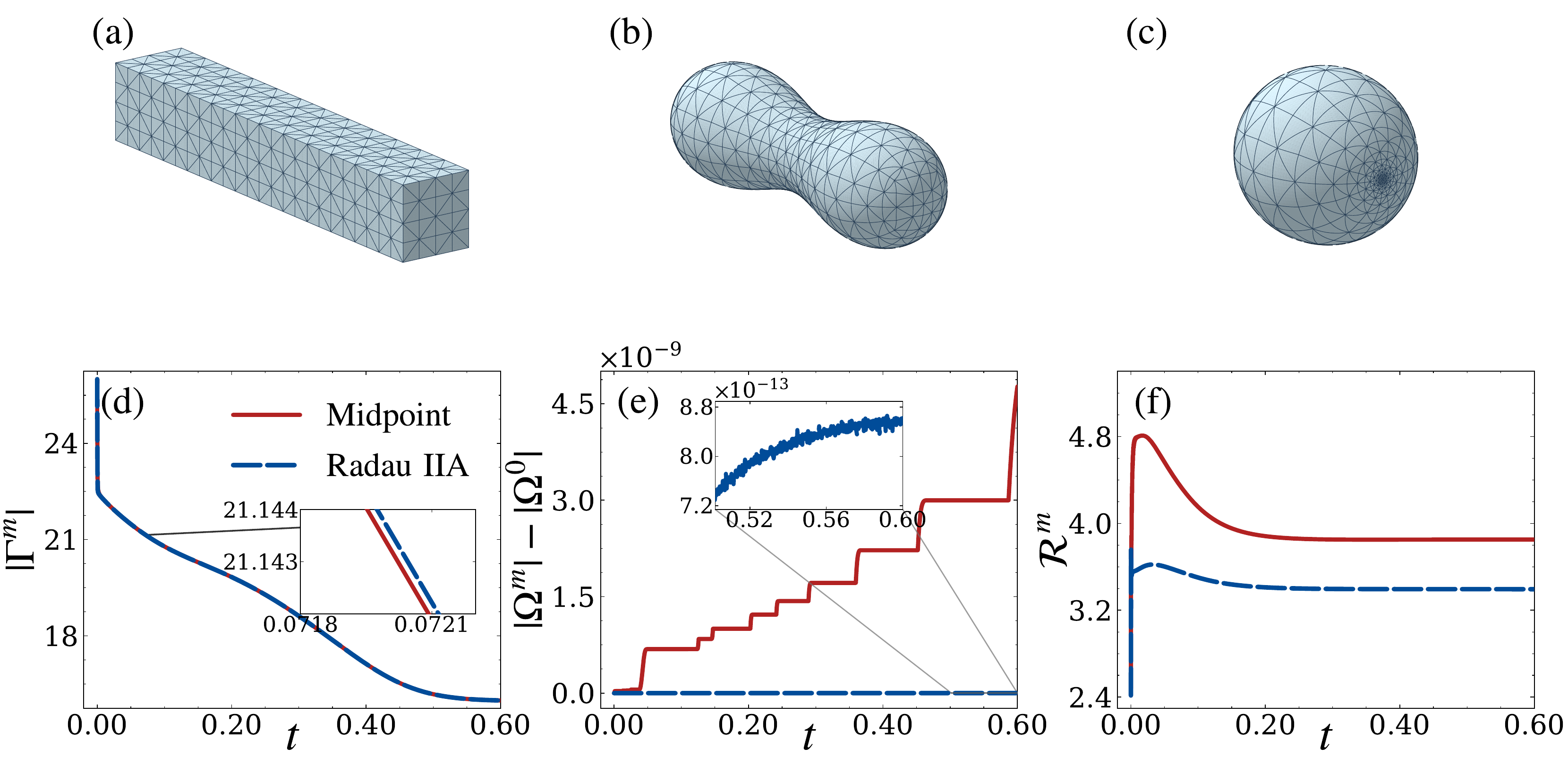}
\caption{SD evolution of a $6\times 1\times 1$ cuboid. (a)--(c): snapshots obtained with Radau IIA at $t=0$, $0.15$, and $0.60$. (d)--(f): surface area $|\Gamma^m|$, enclosed-volume change $|\Omega^m|-|\Omega^0|$, and mesh ratio $R^m$ for Midpoint and Radau IIA.}
\label{fig:cuboid-sd}
\end{figure}

\begin{figure}[!htbp]
\centering
\includegraphics[width=\linewidth]{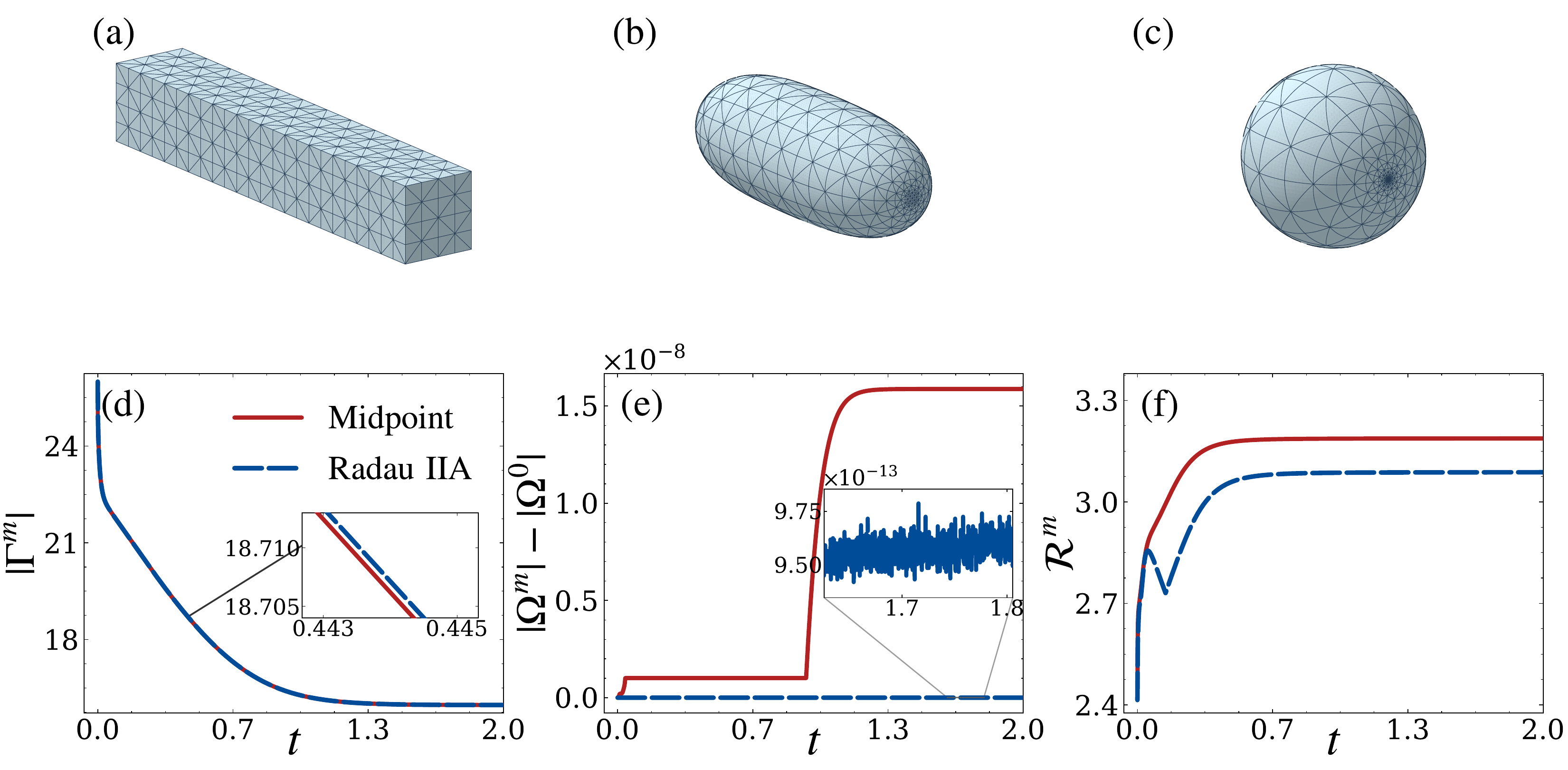}
\caption{VP-MCF evolution of a $6\times 1\times 1$ cuboid. (a)--(c): snapshots obtained with Radau IIA at $t=0$, $0.50$, and $2.00$. (d)--(f): surface area $|\Gamma^m|$, enclosed-volume change $|\Omega^m|-|\Omega^0|$, and mesh ratio $R^m$ for Midpoint and Radau IIA.}
\label{fig:cuboid-vpmcf}
\end{figure}

In Figures~\ref{fig:cuboid-sd}(d)--(e) and~ \ref{fig:cuboid-vpmcf}(d)--(e), the surface-area histories of Midpoint and Radau IIA nearly coincide and decrease monotonically, while the enclosed-volume defects remain of order $10^{-8}$, consistent with the solver tolerance. The mesh-ratio histories in Figures~\ref{fig:cuboid-sd}(f) and~ \ref{fig:cuboid-vpmcf}(f) change during the initial relaxation and then approach a plateau. The two methods give comparable mesh quality under SD, whereas Radau IIA yields lower values at later times for VP-MCF. The plateau indicates that mesh distortion does not continue to accumulate once the initially sharp features have been smoothed.

\subsection{Tests of the equidistribution formulation}
\label{sec:numerical-extension}
We replace the conformal schemes by the
predicted equidistribution schemes \eqref{eq:equi-rk-predicted-scheme} for SD
and \eqref{eq:equi-rk-vpmcf-predicted-scheme} for VP-MCF. In two dimensions,
their Midpoint realizations coincide with
\eqref{eq:direct-rk-predicted-scheme} and
\eqref{eq:conformal-rk-vpmcf-predicted-scheme}, respectively, when the same
stage predictors are used. We therefore use Radau IIA for the planar test
and Midpoint for the surface test.

\paragraph{Example 5: Convergence on the ellipse and ellipsoid}
We repeat the convergence tests for the ellipse with Radau IIA and the
ellipsoid with Midpoint, using the parameters of Examples~1 and~3,
respectively.

\begin{table}[htbp]
\centering
\caption{Hausdorff errors $e_{h,\tau}(T)$ and orders for the ellipse under SD and VP-MCF at $T=0.01$, using the Radau IIA realizations of the predicted equidistribution schemes with $k=2$.}
\label{tab:extension-curve-accuracy}
\small
\begin{tabular}{ccccc}
\hline
 & \multicolumn{2}{c}{SD} & \multicolumn{2}{c}{VP-MCF}\\
\cline{2-3}\cline{4-5}
$(h,\tau)$ & $e_{h,\tau}(T)$ & order & $e_{h,\tau}(T)$ & order\\
\hline
$(h_0,\tau_0)$ & $1.9888\times10^{-2}$ & -- & $2.0258\times10^{-2}$ & --\\
$(h_0/2,\tau_0/2)$ & $1.8969\times10^{-3}$ & 3.390 & $3.5942\times10^{-3}$ & 2.495\\
$(h_0/2^{2},\tau_0/2^{2})$ & $2.5900\times10^{-4}$ & 2.873 & $4.0197\times10^{-4}$ & 3.161\\
$(h_0/2^{3},\tau_0/2^{3})$ & $3.3156\times10^{-5}$ & 2.966 & $5.8217\times10^{-5}$ & 2.788\\
$(h_0/2^{4},\tau_0/2^{4})$ & $4.1806\times10^{-6}$ & 2.987 & $7.5778\times10^{-6}$ & 2.942\\
\hline
\end{tabular}
\end{table}

\begin{table}[htbp]
\centering
\caption{Hausdorff errors $e_{h,\tau}(T)$ and orders for the ellipsoid under SD and VP-MCF at $T=0.005$, using the Midpoint realizations of the predicted equidistribution schemes with $k=1$.}
\label{tab:surface-extension-accuracy}
\small
\begin{tabular}{ccccc}
\hline
 & \multicolumn{2}{c}{SD} & \multicolumn{2}{c}{VP-MCF}\\
\cline{2-3}\cline{4-5}
$(h,\tau)$ & $e_{h,\tau}(T)$ & order & $e_{h,\tau}(T)$ & order\\
\hline
$(h_0,\tau_0)$ & $2.6576\times10^{-1}$ & -- & $2.8351\times10^{-1}$ & --\\
$(h_0/2,\tau_0/2)$ & $6.9351\times10^{-2}$ & 1.938 & $6.8175\times10^{-2}$ & 2.056\\
$(h_0/2^{2},\tau_0/2^{2})$ & $1.8143\times10^{-2}$ & 1.935 & $1.8274\times10^{-2}$ & 1.899\\
$(h_0/2^{3},\tau_0/2^{3})$ & $4.8494\times10^{-3}$ & 1.904 & $4.8707\times10^{-3}$ & 1.908\\
\hline
\end{tabular}
\end{table}

Tables~\ref{tab:extension-curve-accuracy} and
\ref{tab:surface-extension-accuracy} show that Radau IIA exhibits
approximately third-order convergence for the ellipse, while Midpoint
exhibits approximately second-order convergence for the ellipsoid under both
flows. These results are consistent with the corresponding predicted
conformal schemes in Examples~1 and~3 under coupled space--time refinement.

\paragraph{Example 6: Evolution of an ellipsoid}
We evolve the ellipsoid under SD up to $T=0.3$ using the Midpoint realization of \eqref{eq:equi-rk-predicted-scheme}, with time step $\tau=0.005$.

\begin{figure}[!htbp]
\centering
\includegraphics[width=\linewidth]{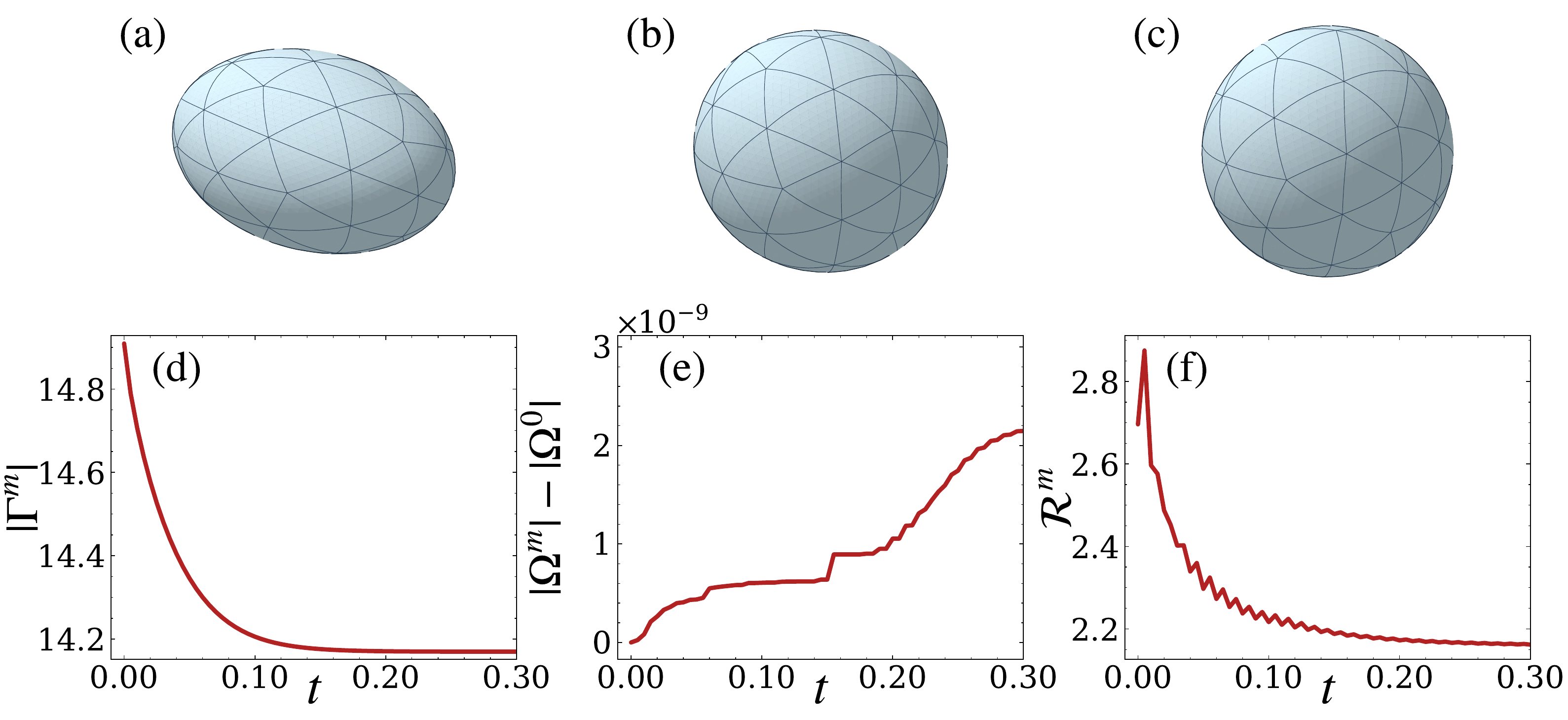}
\caption{SD evolution of the ellipsoid using the Midpoint equidistribution scheme. (a)--(c): snapshots at $t=0$, $0.1$, and $0.3$, respectively. (d)--(f): surface area $|\Gamma^m|$, enclosed-volume change $|\Omega^m|-|\Omega^0|$, and mesh-quality indicator $R^m$.}
\label{fig:extension-ellipsoid-sd}
\end{figure}

Figure~\ref{fig:extension-ellipsoid-sd} confirms the predicted structure-preserving properties: the surface area decreases monotonically, while the enclosed-volume defect remains below $8\times10^{-9}$, consistent with the $10^{-8}$ solver tolerance. The surface mesh remains visually uniform in element size, and the mesh-quality indicator stays bounded and approaches a plateau after a short initial transient.

\section{Conclusions}
\label{sec:conclusions}
The admissible local pullback framework provides a common geometric setting for the high-order discretization of SD and VP-MCF. The classification result shows that, within the considered coordinate-independent and composition-stable class, the available types of admissible local pullbacks are highly restricted: equidistributed local pullbacks are the only option for curves, whereas surfaces admit conformal and equidistributed local pullbacks. The corresponding pullback identities yield the conformal and equidistribution formulations. Combined with isoparametric finite elements and Runge--Kutta methods, both formulations lead to high-order schemes that preserve the enclosed measure and dissipate the geometric energy under their respective conditions. Predicted-coefficient variants reduce the nonlinearity while retaining these structures. The numerical results confirm the expected convergence behavior, structure preservation, and favorable mesh quality. Future work will extend the present framework to the design of high-order structure-preserving methods for a broader class of geometric flows.
\bibliographystyle{siamplain}
\bibliography{references}
\end{document}